\documentclass[12pt]{article}
\usepackage{amsmath,amssymb,amsthm}
\usepackage{graphicx,wrapfig}
\numberwithin{equation}{section}
\theoremstyle{plain}
\newtheorem{theorem}[equation]{Theorem}
\newtheorem{corollary}[equation]{Corollary}
\newtheorem{lemma}[equation]{Lemma}
\newtheorem{proposition}[equation]{Proposition}

\theoremstyle{definition}
\newtheorem{definition}[equation]{Definition}

\newtheorem{example}[equation]{Example}
\newtheorem{remark}[equation]{Remark}

\newcommand*{\hcal}{{\mathcal H}}
\newcommand*{\lcal}{{\mathcal L}}
\newcommand*{\R}{{\mathbb R}}
\newcommand*{\N}{{\mathbb N}}
\newcommand*{\C}{{\mathbb C}}
\newcommand*{\Om}{\Omega}
\newcommand*{\dOm}{{\partial\Omega}}
\newcommand*{\symdiff}{\bigtriangleup}
\newcommand*{\eqcolon}{\mathbin{=:}}
\newcommand*{\coloneq}{\mathbin{:=}}
\providecommand{\vint}[1]{\mathchoice
          {\mathop{\vrule width 5pt height 3 pt depth -2.5pt
                  \kern -9pt \kern 1pt\intop}\nolimits_{\kern -5pt{#1}}}
          {\mathop{\vrule width 5pt height 3 pt depth -2.6pt
                  \kern -6pt \intop}\nolimits_{\kern -3pt{#1}}}
          {\mathop{\vrule width 5pt height 3 pt depth -2.6pt
                  \kern -6pt \intop}\nolimits_{\kern -3pt{#1}}}
          {\mathop{\vrule width 5pt height 3 pt depth -2.6pt
                  \kern -6pt \intop}\nolimits_{\kern -3pt{#1}}}}
\newcommand*{\eps}{\varepsilon}
\newcommand*{\loc}{\mathrm{loc}}
\newcommand*{\BV}{\mathrm{BV}}
\newcommand*{\liploc}{\mathrm{Lip}_{\mathrm{loc}}}
\newcommand*{\ch}{\text{\raise 1.3pt \hbox{$\chi$}\kern-0.2pt}}
\DeclareMathOperator{\capa}{Cap}
\DeclareMathOperator{\dist}{dist}
\DeclareMathOperator{\sgn}{sgn}
\DeclareMathOperator{\diam}{diam}
\DeclareMathOperator{\Lip}{Lip}
\begin{document}
\title{An analog of the Neumann problem for 1-Laplace equation in the metric setting: existence,
boundary regularity, and stability
\footnotetext{
\hfill \break {\it Keywords\,}: bounded variation, metric measure space,
Neumann problem, positive mean curvature, stability}}
\author{Panu Lahti%
\footnote{P.\,L.\@ was supported by the Finnish Cultural Foundation.}
\and Luk\'{a}\v{s} Mal\'{y}%
\footnote{L.\,M.\@ was supported by the Knut and Alice Wallenberg Foundation (Sweden).}
\and Nageswari Shanmugalingam%
\footnote{N.\,S.\@ was partially supported by the grant DMS-1500440 from NSF(U.S.A.).}
}
\maketitle
\begin{abstract}
We study an inhomogeneous Neumann boundary value 
problem for functions of least gradient on bounded domains in metric spaces that are
equipped with a doubling measure and support a Poincar\'e inequality. We show that solutions exist
under certain regularity assumptions on the domain, but are generally nonunique.
We also show that solutions can be taken to be differences of two characteristic functions,
and that they are regular up to the boundary when the boundary is of \emph{positive mean curvature}.
By regular up to the boundary
we mean that if the boundary data is $1$ in a neighborhood of a point on the boundary of the domain, then
the solution is $-1$ in the intersection of the domain with a possibly smaller neighborhood of that point.
Finally, we consider the stability of solutions with respect to boundary data.
\end{abstract}
\section{Introduction}
The goal of the Neumann boundary value problem for $\Delta_p$ in a smooth 
Euclidean domain $\Omega\subset \R^n$ is to find a function
$u\in W^{1,p}(\Om)$ such that
\begin{align*}
\Delta_p u=-\text{div}(|\nabla u|^{p-2}\nabla u)=0&\text{ in }\Om, \text{ and }\\
 |\nabla u|^{p-2}\partial_\eta u=f&\text{ on }\partial\Om,
\end{align*}
where $\partial_\eta u$ is the derivative of $u$ in the direction of outer normal to $\partial\Om$ and
$f\in L^\infty(\partial\Om,\mathcal{H}^{n-1})$ such that $\int_{\partial\Om}f\, d\mathcal{H}^{n-1}=0$. 
For $p=1$ this problem is highly degenerate, see for example~\cite{MRS2,MST}.

In the study of boundary value problems for PDEs, more attention has generally been given to
Dirichlet problems than to Neumann problems. This is especially true in the general setting
of a metric space equipped with a doubling measure that supports a Poincar\'e inequality.
In this setting, nonlinear potential theory for Dirichlet problems when $p>1$ is now well developed,
see the monograph \cite{BB} as well as e.g. \cite{BB-OD,BBS2,BBS3,S2}.
By contrast, Neumann problems have been studied very little. The paper~\cite{DL} dealt
mostly with homogeneous Neumann boundary value problem, while in
the paper~\cite{MaSh}, a Neumann problem was formulated as the minimization of the functional
\[
I_p(u)=\int_{\Om} g_u^p\,d\mu+\int_{\partial\Om} Tu\,f\,dP(\Om,\cdot),
\]
where $g_u$ is an upper gradient of $u$ and $p>1$, see Section~\ref{sec:notation} for notation.
In the Euclidean setting, with $\Om$ a smooth domain, a variant of this boundary value problem was
studied in~\cite{MRS2}, and a connection between the problem for
$p>1$ and the problem for $p=1$ was established through a study of the behavior of solutions $u_p$ for
$p>1$ as $p\to 1^+$. 
For functions $f\in L^\infty(\partial\Om)$, the following norm was associated in~\cite{MRS2, MST}:
\[
\| f\|_*=\sup\bigg \lbrace \frac{\int_{\partial\Om}fw\, dP(\Om,\cdot)}{\| Dw\|(\Om)} :\, 
     w\in \BV(\Om)\text{ with }w\ne 0, \int_{\partial\Om} w\, dP(\Om,\cdot)=0\bigg\rbrace.
\]
The problem of minimizing $I_p$
corresponding to $p=1$ was studied in~\cite{MST} and then in~\cite{MRS2} for Euclidean domains with
Lipschitz boundary, and $\| f\|_*\le1$. The paper~\cite{MRS2} also gave an application of
this problem to the study of electrical conductivity. We point out here that the condition $\| f\|_*\le 1$
gives the minimal energy $I_1(u)=0$, and hence constant functions will certainly minimize the energy.
Our focus in the present paper is to study the situation corresponding to $\| f\|_*>1$, in which case
there are no minimizers for the energy $I_1$ if one seeks to minimize $I_1(u)$ within the class of \emph{all}
functions $u\in \BV(\Om)$, see the discussion in the proof of~\cite[Proposition~3.1]{MRS2}. Thus we are
compelled to add further natural constraints on the competitor functions $u$,
namely that $-1\le u\le 1$. This 
constraint is not as restrictive as it might seem, and instead for any $\beta>0$ we can also consider constraints of the form that
all competitor functions satisfy $-\beta\le u\le \beta$. Then $u_\beta$ is a minimizer for
the constraint that all competitors $v$ should satisfy $-\beta\le v\le \beta$ if and only if $\beta^{-1}u_\beta$
is a minimizer for the constraint that all competitors $v$ satisfy $-1\le v\le 1$. Thus the study undertaken here
complements the results in~\cite{MRS2, MST} in the smooth Euclidean domains setting. For instance, suppose that
$\partial\Om$ is of positive mean curvature (either in the sense of Riemannian geometry in Euclidean setting, or
in the sense of Definition~\ref{defn:positive-curv} in the more general metric setting). For such a domain, whenever
the boundary data $f$ is not $\mathcal{H}$-a.e.~zero on $\dOm$ and
takes on only three values, $-1, 0, 1$, then $\| f\|_*>1$; this interesting setting,
excluded in the studies in~\cite{MRS2,MST}, is covered in Section~5 of the present paper. For an alternate
(but equivalent) framing of the Neumann boundary value problem for $p=1$, see~\cite{Mo17}. The paper~\cite{Mo17} also gives
an application of the problem to the study of conductivity, see~\cite[Section~1.1]{Mo17}. The problem 
as framed in~\cite{Mo17} is not tractable in the metric setting as it relies heavily on the theory of divergence free
$L^\infty$-vector fields, a tool that is lacking in the non-smooth setting.

In this paper, our goal is to study an analogous problem of minimizing $I_p$ 
in the metric setting when $p=1$.
In this case, instead of the $p$-energy it is natural to minimize the total variation 
among functions of bounded variation. 
See e.g. \cite{MRS,MST,SS2,SWZ,ZiZu}
for previous studies of the Dirichlet problem when $p=1$ in the Euclidean setting,
and \cite{HKLS,KLLS,LMSS} in the metric setting.
In this paper, following the formulation given in \cite{MaSh}, we consider minimization of the functional
\[
I(u)=\| Du\|(\Om)+\int_{\partial\Om} Tu\,f\,dP(\Om,\cdot).
\]

Our goal is to study the existence, uniqueness, regularity, and stability properties of
solutions. In Section \ref{sec:prelis} we consider basic properties of solutions and note that they
are generally nonunique. However, in Proposition \ref{prop:existence of set minimizers} we show
that if a solution exists, it can be taken to be of the
form $\ch_{E_1}-\ch_{E_2}$ for disjoint sets $E_1,E_2\subset \Om$. In most of the rest of the
paper, we consider only such solutions. It is clear that these solutions cannot exhibit
much interior regularity, but in Proposition \ref{prop:E1-leastGrad} we show that $\ch_{E_1}$
and $\ch_{E_2}$ are functions of least gradient.

In Section \ref{sec:existence} we show that under some regularity assumptions on $\Om$,
and additionally that $-1\le f\le 1$,
the functional $I(\cdot)$ is lower semicontinuous with respect to convergence in $L^1(\Om)$,
and we use this fact to establish the existence of solutions; this is Theorem
\ref{thm:existence of solutions}.
In Section \ref{sec:when f is integer} we study the boundary regularity of solutions when
$f$ only taken the values $-1,0,1$. In the Euclidean setting, $f$ can be interpreted
as the relative outer normal derivative of the solution, and so one would expect to have
$T\ch_{E_1}=1$ where $f=-1$. This is not always the case, but in
Theorem \ref{thm:poscurv:bdry-data-agreement} we show that $T\ch_{E_1}=1$ in the interior
points of $\{f=-1\}$ when $\Om$ has boundary of \emph{positive mean curvature}.

While solutions are generally nonunique, in Theorem \ref{thm:sol-with-minmeasure} we show that
so-called minimal solutions are unique. Finally, in Section \ref{sec:stability} we study stability
properties of solutions with respect to boundary data, and show that a convergent sequence
of boundary data yields a sequence of solutions that converges up to a subsequence; this is
Theorem \ref{thm:stable1}. Finally, in Theorem~\ref{thm:monotone} we present one method
of explicitly constructing a solution for limit boundary data.

Note that if $\Om$ is a domain such that $\mu(X\setminus\Om)=0$, then as BV functions are
insensitive to sets of measure zero, we will always have that $P(\Om,\cdot)$ is the zero measure and, by
the Poincar\'e inequality,
the minimizer of the functional $I$ is a ($\mu$-a.e.) constant function. This is not a very interesting
situation to consider. The results in this paper will be significant only for domains $\Om$ with $\mu(X\setminus\Om)>0$.
\section{Notation and definitions}
\label{sec:notation}
In this section we introduce the necessary notation and assumptions.

In this paper, $(X,d,\mu)$ is a complete metric space equipped
with a Borel regular outer measure $\mu$ satisfying a doubling property, that is,
there is a constant $C_d\ge 1$ such that
\[
0<\mu(B(x,2r))\le C_d\mu(B(x,r))<\infty
\]
for every ball $B=B(x,r)$ with center $x\in X$ and radius $r>0$.
If a property holds outside a set of $\mu$-measure zero, we say that it holds almost everywhere,
or a.e. We assume that $X$ consists of at least two points. When we want to specify
that a constant $C$ depends on the parameters $a,b, \ldots,$ we write $C=C(a,b,\ldots)$.

A complete metric space with a doubling measure is proper,
that is, closed and bounded subsets are compact. Since $X$ is proper, for any open set $\Omega\subset X$
we define $\liploc(\Omega)$ to be the space of
functions that are Lipschitz in every open $\Omega'\Subset\Omega$.
Here $\Omega'\Subset\Omega$ means that $\overline{\Omega'}$ is a
compact subset of $\Omega$. Other local spaces of functions are defined analogously.

For any set $A\subset X$ and $0<R<\infty$, the restricted spherical Hausdorff content
of codimension $1$ is defined by
\[
\hcal_{R}(A)=\inf\left\{ \sum_{i=1}^{\infty}
  \frac{\mu(B(x_{i},r_{i}))}{r_{i}}:\,A\subset\bigcup_{i=1}^{\infty}B(x_{i},r_{i}),\,r_{i}\le R\right\}.
\]
The codimension $1$ Hausdorff measure of a set $A\subset X$ is given by
\[
  \hcal(A)=\lim_{R\rightarrow 0}\hcal_{R}(A).
\]

The measure theoretic boundary $\partial^*E$ of a set $E\subset X$ is the set of points $x\in X$
at which both $E$ and its complement have positive upper density, i.e.
\[
\limsup_{r\to 0}\frac{\mu(B(x,r)\cap E)}{\mu(B(x,r))}>0\quad
  \textrm{and}\quad\limsup_{r\to 0}\frac{\mu(B(x,r)\setminus E)}{\mu(B(x,r))}>0.
\]
The measure theoretic interior and exterior of $E$ are defined respectively by
\begin{equation}\label{eq:definition of measure theoretic interior}
I_E=\left\{x\in X:\,\lim_{r\to 0}\frac{\mu(B(x,r)\setminus E)}{\mu(B(x,r))}=0\right\}
\end{equation}
and
\begin{equation}\label{eq:definition of measure theoretic exterior}
O_E=\left\{x\in X:\,\lim_{r\to 0}\frac{\mu(B(x,r)\cap E)}{\mu(B(x,r))}=0\right\}.
\end{equation}
A curve $\gamma$ is a nonconstant rectifiable continuous mapping from a compact interval
into $X$.
The length of a curve $\gamma$
is denoted by $\ell_{\gamma}$. We will assume every curve to be parametrized
by arc-length, which can always be done (see e.g. \cite[Theorem 3.2]{Hj}).
A nonnegative Borel function $g$ on $X$ is an upper gradient 
of an extended real-valued function $u$
on $X$ if for all curves $\gamma$ on $X$, we have
\begin{equation}\label{eq:definition of upper gradient}
|u(x)-u(y)|\le \int_0^{\ell_{\gamma}}g(\gamma(s))\,ds,
\end{equation}
where $x$ and $y$ are the end points of $\gamma$. We interpret $|u(x)-u(y)|=\infty$ whenever  
at least one of $|u(x)|$, $|u(y)|$ is infinite.
Upper gradients were originally introduced in~\cite{HK}.

If $g$ is a nonnegative $\mu$-measurable function on $X$
and (\ref{eq:definition of upper gradient}) holds for $1$-a.e. curve,
we say that $g$ is a $1$-weak upper gradient of~$u$. 
A property holds for $1$-a.e. curve
if it fails only for a curve family with zero $1$-modulus. 
A family $\Gamma$ of curves is of zero $1$-modulus if there is a 
nonnegative Borel function $\rho\in L^1(X)$ such that 
for all curves $\gamma\in\Gamma$, the curve integral $\int_\gamma \rho\,ds$ is infinite.

Let $\Omega\subset X$ be open. By only considering curves in $\Om$, we can say that $g$
is an upper gradient of $u$ in $\Om$. 
We let
\[
\| u\|_{N^{1,1}(\Omega)}=\| u\|_{L^1(\Omega)}+\inf \| g\|_{L^1(\Omega)},
\]
where the infimum is taken over all upper gradients $g$ of $u$ in $\Omega$.
The substitute for the Sobolev space $W^{1,1}(\Omega)$ in the metric setting is the Newton-Sobolev space
\[
N^{1,1}(\Omega)\coloneq \{u:\|u\|_{N^{1,1}(\Omega)}<\infty\}.
\]
We understand Newton-Sobolev functions to be defined everywhere (even though
$\| \cdot\|_{N^{1,1}(\Omega)}$ is then only a seminorm).
For more on Newton-Sobolev spaces, we refer to~\cite{S, BB, HKST}.

The $1$-capacity of a set $A\subset X$ is given by
\begin{equation}\label{eq:definition of p-capacity}
 \capa_1(A)=\inf \| u\|_{N^{1,1}(X)},
\end{equation}
where the infimum is taken over all functions $u\in N^{1,1}(X)$ such that $u\ge 1$ in $A$.
We know that when $X$ supports a $(1,1)$-Poincar\'e inequality (see below), $\capa_1$ is an outer capacity, meaning that
\[
\capa_1(A)=\inf\{\capa_1(U):\,U\supset A\textrm{ is open}\}
\]
for any $A\subset X$, see e.g. \cite[Theorem 5.31]{BB}.
If a property holds outside a set
$A\subset X$ with $\capa_1(A)=0$, we say that it holds $1$-quasieverywhere, or $1$-q.e.

Next we recall the definition and basic properties of functions
of bounded variation on metric spaces, following \cite{M}.
See also e.g. \cite{AFP, EvaG92, Fed, Giu84, Zie89} for the classical 
theory in the Euclidean setting.
For $u\in L^1_{\loc}(X)$, we define the total variation of $u$ in $X$ to be 
\[
\|Du\|(X)=\inf\left\{\liminf_{i\to\infty}\int_X g_{u_i}\,d\mu:\, u_i\in \Lip_{\loc}(X),\, u_i\to u\textrm{ in } L^1_{\loc}(X)\right\},
\]
where each $g_{u_i}$ is an upper gradient of $u_i$.
We say that a function $u\in L^1(X)$ is of bounded variation, 
denoted by $u\in\BV(X)$, if $\|Du\|(X)<\infty$.
By replacing $X$ with an open set $\Omega\subset X$ in the definition of the total variation, we can define $\|Du\|(\Omega)$.
For an arbitrary set $A\subset X$, we define
\[
\|Du\|(A)=\inf\{\|Du\|(\Omega):\, A\subset\Omega,\,\Omega\subset X
\text{ is open}\}.
\]
If $u\in\BV(X)$, $\|Du\|(\cdot)$ is a finite Radon measure on $X$ by \cite[Theorem 3.4]{M}.
A $\mu$-measurable set $E\subset X$ is said to be of finite perimeter in $\Om$
if $\|D\ch_E\|(\Om)<\infty$, where $\ch_E$ is the characteristic function of $E$.
The perimeter of $E$ in $\Omega$ is also denoted by
\[
P(E,\Omega)\coloneq\|D\ch_E\|(\Omega).
\]

We have the following coarea formula from~\cite[Proposition 4.2]{M}: if $\Om\subset X$ is
an open set and $u\in \BV(\Om)$, then for any Borel set $A\subset \Om$,
\begin{equation}\label{eq:coarea}
\|Du\|(A)=\int_{-\infty}^{\infty}P(\{u>t\},A)\,dt.
\end{equation}

We will assume throughout that $X$ supports a $(1,1)$-Poincar\'e inequality,
meaning that there exist constants $C_P>0$ and $\lambda \ge 1$ such that for every
ball $B(x,r)$, every locally integrable function $u$ on $X$,
and every upper gradient $g$ of $u$,
we have 
\[
\vint{B(x,r)}|u-u_{B(x,r)}|\, d\mu 
\le C_P r\vint{B(x,\lambda r)}g\,d\mu,
\]
where 
\[
u_{B(x,r)}\coloneq\vint{B(x,r)}u\,d\mu \coloneq\frac 1{\mu(B(x,r))}\int_{B(x,r)}u\,d\mu.
\]
By applying the Poincar\'e inequality to approximating locally Lipschitz functions in the 
definition of the total variation, 
we get the following
for $\mu$-measurable sets $E\subset X$: 
\begin{equation}\label{eq:relative isoperimetric inequality}
\min\{\mu(B(x,r)\cap E),\,\mu(B(x,r)\setminus E)\}\le 2C_P rP(E,B(x,\lambda r)).
\end{equation}

For an open set $\Omega\subset X$ and a $\mu$-measurable set $E\subset X$
with $P(E,\Omega)<\infty$, we know that for any Borel set $A\subset\Om$,
\begin{equation}\label{eq:def of theta}
P(E,A)=\int_{\partial^*E\cap A}\theta_E\,d\hcal,
\end{equation}
where
$\theta_E\colon X\to [\alpha,C_d]$ with $\alpha=\alpha(C_d,C_P,\lambda)>0$, see \cite[Theorem 5.3]{A1} 
and \cite[Theorem 4.6]{AMP}.

The lower and upper approximate limits of a function $u$ on $X$ are defined respectively by
\[
u^{\wedge}(x):
=\sup\left\{t\in\R:\,\lim_{r\to 0}\frac{\mu(B(x,r)\cap\{u<t\})}{\mu(B(x,r))}=0\right\}
\]
and
\[
u^{\vee}(x):
=\inf\left\{t\in\R:\,\lim_{r\to 0}\frac{\mu(B(x,r)\cap\{u>t\})}{\mu(B(x,r))}=0\right\}.
\]
The jump set of a function $u$ is the set
\[
S_{u}\coloneq \{x\in X:\, u^{\wedge}(x)<u^{\vee}(x)\},
\]

By \cite[Theorem 5.3]{AMP}, the variation measure of a $\BV$ function
can be decomposed into the absolutely continuous and singular part, and the latter
into the Cantor and jump part, as follows. Given an open set 
$\Omega\subset X$ and $u\in\BV(\Omega)$, we have for any Borel set $A\subset X$
\begin{align}
\notag
\| Du\|(A) &=\| Du\|^a(A)+\| Du\|^s(A)\\
\label{eq:decomposition}
&=\| Du\|^a(A)+\| Du\|^c(A)+\| Du\|^j(A)\\
\notag
&=\int_{A}a\,d\mu+\| Du\|^c(A)+\int_{A\cap S_u}\int_{u^{\wedge}(x)}^{u^{\vee}(x)}\theta_{\{u>t\}}(x)\,dt\,d\hcal(x),
\end{align}
where $a\in L^1(\Omega)$ is the density of the absolutely continuous part and the functions $\theta_{\{u>t\}}$ 
are as in~\eqref{eq:def of theta}.
\begin{definition} 
Let $\Omega\subset X$ be an open set and let $u$ be a $\mu$-measurable function on $\Omega$. 
For $x\in\partial\Omega$, the number $Tu(x)$ is the \emph{trace} of $u$ if
\[
\lim_{r\to 0}\,\vint{B(x,r)\cap\Omega}|u-Tu(x)|\,d\mu=0.
\]
\end{definition}
It is straighforward to check that the trace is always a Borel function on the set where it exists.
\begin{definition}
Let $\Om\subset X$ be an open set. A function $u\in \BV_{\loc}(\Om)$
is said to be of \emph{least gradient in $\Om$} if
\[
\| Du\|(\Om)\le \| D(u+\varphi)\|(\Om)
\]
for every $\varphi\in \BV(\Om)$ with compact support in $\Om$.
\end{definition}
\section{Preliminary results}
\label{sec:prelis}
In this section we define the Neumann problem and consider various basic properties of
solutions.

In this section, we always assume that $\Omega\subset X$ is a nonempty bounded open set
with $P(\Om,X)<\infty$, such that
for any $u\in \BV(\Omega)$, the trace $Tu(x)$
exists for $\hcal$-a.e.\@ $x\in\partial^*\Omega$ and thus also for 
$P(\Omega,\cdot)$-a.e.\@ $x\in\partial^*\Omega$, by \eqref{eq:def of theta}. 
See \cite[Theorem 3.4]{LaSh} for conditions on $\Omega$ that guarantee that this holds.

For some of our results, we will also assume that the following exterior measure
density condition holds:
\begin{equation}\label{eq:exterior measure density condition}
\limsup_{r\to 0}\frac{\mu(B(x,r)\setminus \Omega)}{\mu(B(x,r))}>0
\quad \textrm{for }\hcal\text{-a.e.\@ }x\in \partial\Omega.
\end{equation}
Moreover, in this section we always assume that $f\in L^1(\partial^*\Omega,P(\Om,\cdot))$ such that
\begin{equation}\label{eq:f integrates to zero}
\int_{\partial^*\Omega}f\,d P(\Omega,\cdot)=0.
\end{equation}

Throughout this paper we will consider the following functional:
for $u\in\BV(\Omega)$, let
\[
I(u)=\| Du\|(\Omega)+\int_{\partial^*\Omega}Tu\, f\,dP(\Omega,\cdot).
\]
First we note the following basic property of the functional. We denote
$u_+=\max\{u,0\}$ and $u_-=\max\{-u,0\}$.
\begin{lemma}\label{lem:splitting into positive and negative parts}
For any $u\in\BV(\Omega)$, we have $I(u)=I(u_+)+I(-u_-)$.
\end{lemma}
\begin{proof}
Note that for any $\mu$-measurable $E\subset X$, we have $P(E,\Omega)=P(\Omega\setminus E,\Omega)$.
Since $\mu$ is $\sigma$-finite on $X$, it follows that for $\mathcal L^1$-a.e.\@ $t\in \R$
we have $\mu(\{u=t\})=0$ and thus $P(\{u<t\},\Omega)=P(\{u\le t\},\Omega)$,
where $\mathcal L^1$ is the Lebesgue measure.
Thus by the $\BV$ coarea formula \eqref{eq:coarea}, we have
\begin{align*}
I(u)
&=\int_{-\infty}^\infty P(\{u>t\},\Omega)\,dt+\int_{\partial^*\Omega}Tu\,f\,dP(\Omega,\cdot)\\
& =\int_{0}^\infty P(\{u>t\},\Omega)\,dt+\int_{-\infty}^0 P(\{u< t\},\Omega)\,dt+\int_{\partial^*\Omega}Tu\,f\,dP(\Omega,\cdot)\\
& =\int_{-\infty}^\infty P(\{u_+>t\},\Omega)\,dt+\int_{-\infty}^\infty P(\{u_->t\},\Omega)\,dt
+\int_{\partial^*\Omega}Tu\,f\,dP(\Omega,\cdot)\\
&=\| Du_+\|(\Omega)+\int_{\partial^*\Omega}Tu_+\,f\,dP(\Omega,\cdot)
+\| Du_-\|(\Omega)-\int_{\partial^*\Omega}Tu_-\,f\,dP(\Omega,\cdot)\\
&=I(u_+)+I(-u_-). \qedhere
\end{align*}
\end{proof}
Note that for $u\equiv 0$, $I(u)=0$. Thus, if $I(u)\ge 0$ for all $u\in\BV(\Omega)$, then
we find a minimizer simply by taking the zero function. Hence, we are more interested in the case 
where $I(u)<0$ for some $u\in\BV(\Omega)$. But then
\[
\lim_{\beta\to \infty} I(\beta u)=\lim_{\beta\to \infty} \beta I(u)=-\infty.
\]
Thus, we consider the following restricted minimization problem.
\begin{definition}\label{restricted minimization problem}
We say that a function $u\in \BV(\Omega)$ 
solves the restricted Neumann boundary value problem with boundary data $f$ if $-1\le u\le 1$ 
and $I(u)\le I(v)$ for all $v\in \BV(\Omega)$ with $-1\le v\le 1$.
\end{definition}
The restricted problem does not always have a solution. It may also have only trivial, i.e., constant, 
solutions even though the boundary data are non-trivial. Moreover, non-trivial solutions need not be unique.
In the Euclidean setting these issues were observed in~\cite{MRS2}.
\begin{example}\label{ex:nonexistence of solution}
In the unweighted plane (endowed with the Euclidean distance), consider the unit square, i.e., 
$\Omega=(0,1)^2$. Fix a constant $a>0$ and let $f=-a$ on the middle third portion of the bottom 
side, $f=a$ on the middle third portion of the top side, and $f=0$ elsewhere on the boundary. 
If $a>1$, then
    \[
    \inf_{u\in\BV(\Omega),\,\| u\|_{L^{\infty}(\Om)\le 1}} I(u)= \frac{2(1-a)}{3},
    \]
    but no admissible function gives this infimum.
\end{example}
\begin{example}\label{ex:nonuniqueness of solution}
Consider again the unit square $\Omega=(0,1)^2$ in the Euclidean plane. Fix a constant $a>0$ and let 
$f=-a$ on the bottom side, $f=a$ on the top side, and $f=0$ on the vertical sides. 
\renewcommand*{\theenumi}{\alph{enumi}}
\renewcommand*{\labelenumi}{(\theenumi)}
\begin{enumerate}
  \item If $a \in (0, 1)$, then $\inf_{u} I(u)=0$,
    which is attained only by $u \equiv c$ for any constant $c \in [-1, 1]$. The fact that no other solutions exist can be proven using Proposition~\ref{prop:existence of set minimizers} and Proposition~\ref{prop:E1-leastGrad} below.
  \item If $a=1$, then $\inf_{u} I(u)=0$, which is attained by any constant function $u\equiv c$ with $c\in[-1,1]$ as well as by any function $u(x,y) = v(y)$, $(x,y)\in\Om$, where $v$ is an arbitrary decreasing function with $v(0^+) = 1$, $v(1^-)=-1$.
\end{enumerate}
\end{example}
See also Example \ref{exa:minimality-lost2} for an example of nonuniqueness with $I(u)<0$.

Next we will show that it suffices to consider only a special subclass of $\BV$ functions
as candidates for a solution to the restricted Neumann problem.
First we note that we have the following version of Cavalieri's principle, which can be obtained
from the usual Cavalieri's principle by decomposing $\nu$ into its positive and negative parts.
\begin{lemma}
Let $\nu$ be a signed Radon measure on $X$. Then, for any nonnegative
$h\in L^1(X,|\nu|)$,
\[
\int_X h\,d\nu=\int_0^{\infty}\nu(\{h>t\})\,dt.
\]
\end{lemma}
\begin{proposition}\label{prop:existence of set minimizers}
Let $u\in\BV(\Omega)$ with $-1\le u\le 1$. Then, there exist
disjoint $\mu$-measurable sets $E_1,E_2 \subset \Omega$ such that
\[
I(\ch_{E_1}-\ch_{E_2})\le I(u).
\]
Furthermore, if $u$ is a solution to the restricted Neumann problem 
with boundary data $f$, then for $\lcal^1$-a.e.~$t_1,t_2\in (0,1)$,
the sets
\[
E_1\coloneq \{x\in\Om :\, u(x)>t_1\}\quad\text{and}\quad E_2\coloneq \{x\in\Om :\, u(x)<-t_2\}
\]
give a solution $\ch_{E_1}-\ch_{E_2}$ to the same restricted Neumann problem.
\end{proposition}
\begin{proof}
By Lemma \ref{lem:splitting into positive and negative parts} we have $I(u)=I(u_+)+I(-u_-)$.
By using the $\BV$ coarea formula \eqref{eq:coarea}, and applying the above Cavalieri's
principle with $d\nu=f\,dP(\Omega,\cdot)$,
\begin{equation}\label{eq:caval}
I(u_+)=\int_0^1 \left( P(\{u_+>t\},\Omega) + \int_{\{T u_+>t\}}f\, dP(\Omega,\cdot)\right)\,dt.
\end{equation}
If $t\in (0,1)$ and $Tu_+(x)<t$ for some $x\in\partial^*\Omega$, then 
\begin{align*}
\limsup_{r\to 0}\,\vint{B(x,r)\cap \Omega} & |\ch_{\{u_+>t\}}-0|\,d\mu
=\limsup_{r\to 0}\frac{\mu(B(x,r)\cap \{u_+>t\})}{\mu(B(x,r)\cap \Omega)}\\
&\le\frac{1}{t-Tu_+(x)}\limsup_{r\to 0}\,\vint{B(x,r)\cap \Omega}|u_+-Tu_+(x)|\,d\mu
=0.
\end{align*}
Thus, $T\ch_{\{u_+>t\}}(x)>0$ yields that $Tu_+(x)\ge t$.
Conversely, we see that if $Tu_+(x)>t$, then $T\ch_{\{u_+>t\}}(x)=1$.
In conclusion,
\[
\ch_{\{Tu_+> t\}} \le T\ch_{\{u_+>t\}}\le \ch_{\{Tu_+\ge t\}}.
\]
However, $P(\Omega,\{Tu_+=t\})=0$ for $\mathcal L^1$-a.e. $t\in (0,1)$,  since
$P(\Omega,\cdot)$ is a finite measure.
Thus \eqref{eq:caval} becomes
\[
I(u_+)= \int_0^1 \left( P(\{u_+>t\},\Omega) +
\int_{\partial^*\Omega} T\ch_{\{u_+>t\}}\,f\, dP(\Omega,\cdot)\right)\,dt,
\]
that is,
\begin{equation}\label{eq:caval2}
I(u_+)= \int_0^1 I(\ch_{\{{u_+}>t\}})\,dt.
\end{equation}
Thus there is $t_1\in (0,1)$ such that
$I(\ch_{\{{u_+}>t_1\}})\le I(u_+)$,
which is the same as
$I(\ch_{\{{u}>t_1\}})\le I(u_+)$.

Denoting $I(\cdot)=I_f(\cdot)$ to make the dependence on $f$ explicit,
with the substitutions of $f$ by $-f$ and $u_+$ by $u_-$, inequality~\eqref{eq:caval2}
becomes
\begin{equation}\label{eq:caval3}
I_{f}(-u_-)=I_{-f}(u_-)= \int_0^1 I_{-f}(\ch_{\{{u_-}>t\}})\,dt
=\int_0^1 I_{f}(-\ch_{\{{u_-}>t\}})\,dt.
\end{equation}
Therefore, there is $t_2\in (0,1)$ such that
$I_f(-\ch_{\{u_->t_2\}})\le I_{f}(-u_-)$,
i.e., $I_f(-\ch_{\{u<-t_2\}})\le I_{f}(-u_-)$. Letting $E_1=\{u>t_1\}$ and
$E_2=\{u<-t_2\}$, we now have by Lemma \ref{lem:splitting into positive and negative parts}
\[
I(\ch_{E_1}-\ch_{E_2})=I(\ch_{E_1})+I(-\ch_{E_2})\le I(u_+)+I(-u_-)=I(u),
\]
proving the first claim.

Now let $u$ be a solution, and $t_1$ and $t_2$ as above.
If we had $I(\ch_{\{u>s\}})<I(u_+)$ for some $s\in (0,1)$, then
\[
I(\ch_{\{u>s\}})+I(-\ch_{\{u<-t_2\}})<I(u),
\]
which contradicts $u$ being a solution. Thus from \eqref{eq:caval2} it follows that
$I(\ch_{\{u>s\}})=I(u_+)$ for $\mathcal L^1$-a.e. $s\in (0,1)$.
Analogously, using \eqref{eq:caval3}
we find that $I(-\ch_{\{u<-s\}})=I(-u_-)$ for $\mathcal L^1$-a.e. $s\in (0,1)$,
and this proves the second claim.
\end{proof}
\begin{lemma}\label{lem:I(E)=I(-E)}
If $E\subset\Omega$ is of finite perimeter in $\Om$, then
$I(-\ch_{\Om\setminus E})=I(\ch_E)$.
Thus, if $E_1,E_2\subset \Omega$ are disjoint sets such that $\ch_{E_1}-\ch_{E_2}$ solves the restricted Neumann problem, then necessarily $I(\ch_{E_1})=I(-\ch_{E_2})$,
and $\ch_{E_1}-\ch_{\Om\setminus E_1}$ is also a solution.
\end{lemma}
\begin{proof}
If $P(E,\Om)<\infty$, note that $P(\Om\setminus E,\Omega)=P(E,\Omega)$, and that
for $\hcal$-a.e.\@ $x\in\partial^*\Omega$,
\[
T\ch_{E}(x)+T\ch_{\Om\setminus E}(x)= T(\ch_{E}(x)+\ch_{\Om\setminus E}(x))=T\ch_\Om(x)=1.
\]
Thus,
\begin{align*}
I(-\ch_{\Om\setminus E})
& =P(\Om\setminus E,\Omega)-\int_{\partial^*\Omega}T\ch_{\Om\setminus E} \,f\,dP(\Omega,\cdot)\\
&  =P(E,\Omega)-\int_{\partial^*\Omega}T\ch_{\Om\setminus E} \,f\,dP(\Omega,\cdot)\\
&  =P(E,\Omega)-\int_{\partial^*\Omega}T\ch_{\Om\setminus E} \,f\,dP(\Omega,\cdot)+\int_{\partial^*\Omega}\,f\,dP(\Omega,\cdot)\quad\textrm{by }
\eqref{eq:f integrates to zero}\\
&  =P(E,\Omega)+\int_{\partial^*\Omega}T\ch_{E}\,f\,dP(\Omega,\cdot)\\
&  =I(\ch_E).
\end{align*}
Next, let $E_1,E_2\subset \Omega$ be disjoint sets such that $\ch_{E_1}-\ch_{E_2}$ solves the restricted Neumann problem. If $I(\ch_{E_1})<I(-\ch_{E_2})$, then by
the above, we also have $I(-\ch_{\Om\setminus E_1})<I(-\ch_{E_2})$. Then, by
Lemma~\ref{lem:splitting into positive and negative parts},
\[
I(\ch_{E_1}-\ch_{\Om\setminus E_1}) 
= I(\ch_{E_1})+I(-\ch_{\Om\setminus E_1})
< I(\ch_{E_1})+I(-\ch_{E_2})
=I(\ch_{E_1}-\ch_{E_2}),
\]
a contradiction. Similarly, $I(\ch_{E_1})>I(-\ch_{E_2})$ is impossible.
Moreover, now
\begin{align*}
I(\ch_{E_1}-\ch_{\Om\setminus E_1})
& =I(\ch_{E_1})+I(-\ch_{\Om\setminus E_1}) 
\\
& =2I(\ch_{E_1})  =I(\ch_{E_1})+I(-\ch_{ E_2})
=I(\ch_{E_1}-\ch_{E_2}),
\end{align*}
so that $\ch_{E_1}-\ch_{\Om\setminus E_1}$ is also a solution.
\end{proof}
\begin{lemma}\label{lem:both E1 and E2 are minimizers}
Let $E_1,E_2\subset \Om$ be disjoint
$\mu$-measurable sets.
Then, $\ch_{E_1}-\ch_{E_2}$ solves the restricted Neumann problem if and only if
\[
I(\ch_{E_1})\le I(\ch_F)\quad\textrm{and}\quad I(-\ch_{E_2})\le I(-\ch_F)
\]
for all $\mu$-measurable sets $F\subset\Om$.
\end{lemma}
\begin{proof}
Suppose that $\ch_{E_1}-\ch_{E_2}$ is a solution.
If there is a $\mu$-measurable set $F\subset \Om$ with $I(\ch_F)<I(\ch_{E_1})$, then
by Lemma \ref{lem:splitting into positive and negative parts} and
Lemma \ref{lem:I(E)=I(-E)}
\begin{align*}
I(\ch_F-\ch_{\Om\setminus F})
&=I(\ch_F)+I(-\ch_{\Om\setminus F})\\
&=2I(\ch_F)
<2I(\ch_{E_1})=I(\ch_{E_1})+I(-\ch_{E_2})=I(\ch_{E_1}-\ch_{E_2}),
\end{align*}
a contradiction. Similarly, $I(-\ch_{F})<I(-\ch_{E_2})$ is impossible.

If $E_1,E_2\subset\Omega$ are such that $I(\ch_{E_1})\le I(\ch_F)$ and
$I(-\ch_{E_2})\le I(-\ch_F)$ for all $\mu$-measurable sets $F\subset\Om$, then
\[
I(\ch_{E_1}-\ch_{E_2})=I(\ch_{E_1})+I(-\ch_{E_2})\le I(\ch_{F_1})+I(-\ch_{F_2})
=I(\ch_{F_1}-\ch_{F_2})
\]
for any two disjoint $\mu$-measurable sets $F_1,F_2\subset\Om$. In view of
Proposition~\ref{prop:existence of set minimizers}, $\ch_{E_1}-\ch_{E_2}$ must be a solution.
\end{proof}
Recall that a function $u\in \BV(\Om)$ is of least gradient in $\Om$ if
\[
 \| Du\|(\Om)\le \| D(u+\varphi)\|(\Om)
\]
for every $\varphi\in \BV(\Om)$ with compact support in $\Om$.
\begin{proposition}\label{prop:E1-leastGrad}
Let $E_1, E_2\subset\Om$ be disjoint sets such that $\ch_{E_1}-\ch_{E_2}$ solves the restricted Neumann problem.
Then, $\ch_{E_1}$ and $\ch_{E_2}$ are functions of least gradient in $\Om$.
\end{proposition}
\begin{proof}
To show that $\ch_{E_1}$ is a function of least gradient, it suffices to show that $\| D\ch_{E_1}\|(\Om)\le \| D\ch_F\|(\Om)$ whenever $F\subset\Om$
is a $\mu$-measurable set with $F\symdiff E_1\Subset\Om$,
see~\cite[Lemma~3.2]{KKLS}. Let $F$ be such a set. By Lemma~\ref{lem:both E1 and E2 are minimizers}, 
$I(\ch_{E_1})\le I(\ch_F)$.
On the other hand, $\ch_{E_1}=\ch_F$ in a neighborhood of $\partial\Om$. It follows that
\[
\| D\ch_{E_1}\|(\Om)\le \| D\ch_F\|(\Om),
\]
so that $\ch_{E_1}$ is of least gradient. The proof for $E_2$ is analogous.
\end{proof}
The above is our main result on the interior regularity of solutions; from the proposition it
follows that the sets $E_1,E_2$ and their complements are porous in $\Om$, see
\cite[Theorem 5.2]{KKLS}.

Since solutions can be constructed from sets $E$ of finite perimeter in $\Omega$ and 
since $\Omega$ is itself of finite perimeter in $X$, it is useful to know that
the sets $E$ are also of finite perimeter in $X$.
\begin{theorem}[{\cite[Corollary 6.13]{KLLS}}]\label{thm:extendability of sets of finite perimeter}
Assume that $\widehat{\Omega}\subset X$ is a bounded open set with $P(\widehat{\Omega},X)<\infty$,
and suppose that there exists
$N\subset \partial\widehat{\Omega}$ with $\hcal(N)<\infty$ such that
\[
\limsup_{r\to 0}\frac{\mu(B(x,r)\setminus\widehat{\Omega})}{\mu(B(x,r))}>0
\]
for every $x\in \partial\widehat{\Omega}\setminus N$.
Let $E\subset \widehat{\Omega}$ such that $P(E,\widehat{\Omega})<\infty$. Then 
$E$ is of finite perimeter in $X$.
\end{theorem}
Note that if $\widehat{\Om}$ satisfies the condition listed in~\eqref{eq:exterior measure density condition}, 
then $\hcal(N)=0$ above. 
\begin{lemma}\label{lem:coincidence of perimeter of E and Omega}
Assume that $\Om$ satisfies the exterior measure density condition
\eqref{eq:exterior measure density condition}.
Let $E\subset \Omega$ be a $\mu$-measurable set with $P(E,X)<\infty$. Then, for any
Borel set $A\subset \partial\Om$, we have
\[
P(E,A)=P(\Omega,A\cap \{T\ch_E=1\}).
\]
\end{lemma}
By Theorem \ref{thm:extendability of sets of finite perimeter}, we can equally well only
assume that $P(E,\Om)<\infty$.
\begin{proof}
Note that the trace $T\ch_E(x)$ is defined for $\mathcal H$-a.e. $x\in\partial^*\Om$ and
can only take the values $0$ and $1$. Also, $P(E,\cdot)$ is concentrated on $\partial^*E$, and
$(\partial\Om\setminus\partial^*\Om)\cup\{T\ch_E=0\}\subset O_E$, i.e., the measure
theoretic exterior of $E$ as defined by \eqref{eq:definition of measure theoretic exterior}.
Thus, we have
\begin{equation}\label{eq:perimeter concentrated on trace one set}
P(E,\partial\Omega\setminus \{T\ch_E=1\})\le
P(E,\partial\Om\setminus\partial^*\Om)+P(E,\partial^*\Omega\cap \{T\ch_E=0\})
=0.
\end{equation}
Consider a point $x\in\partial\Omega$ where $T\ch_E(x)=1$. We have
\[
\frac{\mu(B(x,r)\cap (\Omega \symdiff E))}{\mu(B(x,r)}\to 0\quad\textrm{as }r\to 0,
\]
so that $x\in\partial^*E$ if and only if $x\in\partial^*\Om$, and
according to \cite[Proposition 6.2]{AFP}, for $\hcal$-a.e.\@ such point  we have $\theta_{E}(x)=\theta_{\Omega}(x)$; recall \eqref{eq:def of theta}.
In total, by \eqref{eq:perimeter concentrated on trace one set} and by applying
\eqref{eq:def of theta} twice,
\begin{align*}
P(E,A)
 =P(E,A\cap \{T\ch_E=1\})
 & =\int_{A\cap \{T\ch_E=1\}\cap \partial^*E}\theta_E\,d\hcal \\
 =\int_{A\cap \{T\ch_E=1\}\cap \partial^*\Omega}\theta_{\Omega}\,d\hcal
& =P(\Omega,A\cap \{T\ch_E=1\}).
\qedhere
\end{align*}
\end{proof}
\begin{lemma}\label{lem:estimate for perimeter in whole space}
Suppose that $\Om$ satisfies the exterior measure density condition
\eqref{eq:exterior measure density condition}, and that $-1\le f\le 1$.
Let $E\subset \Om$ be of finite perimeter in $\Om$. Then
\[
P(E,X)\le  I(\ch_{E})+2P(\Omega,X).
\]
\end{lemma}
\begin{proof}
By Theorem \ref{thm:extendability of sets of finite perimeter}, $P(E,X)<\infty$.
By the definition of the functional and the fact that $-1\le f\le 1$,
\begin{align*}
P(E,\Omega)\le I(\ch_{E})+P(\Omega,X),
\end{align*}
whereas by Lemma \ref{lem:coincidence of perimeter of E and Omega},
$P(E,\partial\Omega)\le P(\Omega,X)$.
Thus we get
\[
P(E,X)=P(E,\Omega)+P(E,\partial\Omega)\le  I(\ch_{E})+2P(\Omega,X). \qedhere
\]
\end{proof}
\section{Existence of solutions}
\label{sec:existence}
In this section, we prove that under fairly mild assumptions on $\Om$,
solutions to the restricted Neumann problem given on page
\pageref{restricted minimization problem} exist. This is Theorem~\ref{thm:existence of solutions}.

We say that a set $A\subset X$ is \emph{$1$-quasiopen} if for every $\eps>0$ there is an
open set $G\subset X$ with $\capa_1(G)<\eps$ such that $A\cup G$ is open.
Note that $1$-quasiopen sets do not in general form a topology: as is noted in \cite{BBM}, all 
singletons in unweighted $\R^n$, $n\ge 2$, are $1$-quasiopen, but not all sets are $1$-quasiopen. 
Nonetheless, countable unions as well as finite intersections of $1$-quasiopen sets are $1$-quasiopen 
by~\cite[Lemma~2.3]{Fug71}.

The following lemma is well known in the Euclidean setting,
and has been proved in the metric setting in \cite[Lemma 3.8]{L-SA}.
\begin{lemma}\label{lem:variation measure and capacity}
	Let $\Omega\subset X$ be an open set, and
	let $u\in L^1_{\loc}(\Omega)$ with $\| Du\|(\Omega)<\infty$.
	Then for every $\eps>0$ there exists $\delta>0$ such that
	if $A\subset \Omega$ with $\capa_1(A)<\delta$, then $\| Du\|(A)<\eps$.
\end{lemma}
From this lemma it easily follows that $1$-quasiopen sets are always $\| Du\|$-measurable,
and we will use this fact without further notice.

The total variation is easily seen to be lower semicontinuous with respect to $L^1$-convergence in any open set. We will need the following more general semicontinuity result
that follows from \cite[Theorem 4.5]{L2}.
\begin{proposition}\label{prop:lower semicontinuity in quasiopen sets}
Let $u\in L^1_{\loc}(X)$ such that $\| Du\|(X) <\infty$,
and suppose that $u_i\to u$ in $L^1_{\loc}(X)$. Then, for every $1$-quasiopen set $U\subset X$, we have
\[
\| Du\|(U)\le\liminf_{i\to\infty}\| Du_i\|(U).
\]
\end{proposition}
To deal with boundary values given by a function $f$ defined only on $\partial\Omega$,
we first need to extend $f$ to the whole space in a suitable way.
We will consider open sets $\Om$ whose boundary is codimension $1$ Ahlfors regular in the following
sense: for every $x\in\partial\Omega$, every $0<r\le\diam(\Omega)$, and some constant $C_A\ge 1$,
\begin{equation}\label{eq:boundary codim Ahlfors regularity}
\frac{1}{C_A}\frac{\mu(B(x,r))}{r}\le \hcal(B(x,r)\cap 
\partial\Omega)\le C_A\frac{\mu(B(x,r))}{r}.
\end{equation}
\begin{theorem}\label{thm:extension from boundary}
Let $\Omega\subset X$ be a bounded open set whose boundary is codimension $1$ Ahlfors regular
as given in \eqref{eq:boundary codim Ahlfors regularity}.
Let $f\in L^1(\partial\Omega,\hcal)$ with $-1\le f\le 1$.
Then, there exists $Ef\in N^{1,1}(X\setminus \partial\Omega)$ with $-1\le Ef\le 1$ and
\[
\lim_{r\to 0}\,\vint{B(x,r)}|Ef-f(x)|\,d\mu=0
\]
for $\hcal$-a.e. $x\in\partial\Omega$.
\end{theorem}
\begin{proof}
This follows from \cite{MSS} and the proofs therein. Note that the argument of the extension
theorem for Besov boundary data \cite[Theorem~1.1]{MSS} needs to be slightly modified to produce
a Newtonian extension not only inside $\Om$ but in the whole set $X \setminus \dOm$. Namely,
when constructing a Whitney-type decomposition $\mathcal{W}_{X \setminus \dOm}$, we
consider only balls whose distance from $\dOm$ is at most $2 \diam( \Om)$. Such a collection of
balls covers $\Om$ as well as the $2 \diam(\Om)$-neighborhood of $\dOm$, leaving
out $\dOm$.
Then, we relax the requirements on the partition of unity $\{\phi_{j,i}\}_{j,i}$ subordinate
to $\mathcal{W}_{X \setminus \dOm}$ by demanding that
\begin{align*}
  \sum_{j,i} \phi_{j,i}(x) &= 1 \quad \text{for } x \in X \setminus \dOm,\ \dist(x, \dOm)\le \diam(\Om), \quad \text{and}\\
  \sum_{j,i} \phi_{j,i}(x) &\le 1 \quad \text{for } x \in X \setminus \dOm,\ \dist(x, \dOm)> \diam(\Om).
\end{align*}
Using such a ``partition of unity'' gives us an extension of $f$ in the class
$N^{1,1}(X\setminus \dOm) \cap \Lip_\loc(X \setminus \dOm)$ such that this extension vanishes 
outside of the $3 \diam(\Om)$-neighborhood of $\dOm$.

The extension theorem for Besov boundary data modified as described above can then be used directly 
in~\cite[Theorem 1.2]{MSS} to find the desired extension 
$Ef \in N^{1,1}(X\setminus \dOm) \cap \Lip_\loc(X \setminus \dOm)$ for any $L^1$-boundary 
data $f\colon \dOm \to [-1,1]$, truncating $Ef$ at levels $\pm 1$ if needed.
\end{proof}
Note that for any $A\subset X$, by \cite[Theorem~4.3, Theorem~5.1]{HaKi} 
(see also~\cite[Corollary~5.3]{HaSh}) we have
\begin{equation}\label{eq:null sets of Hausdorff measure and capacity}
\hcal(A)=0\quad\textrm{if and only if}\quad \capa_1(A)=0.
\end{equation}

In the following, given a ball $B=B(x,r)$ we sometimes abbreviate $2B\coloneq B(x,2r)$.
\begin{proposition}\label{prop:extending to the boundary as Newtonian function}
Let $\Omega\subset X$ be a bounded open set with $\hcal(\partial\Omega)<\infty$,
and let $f$ be a function on $X$ such that
$-1\le f\le 1$, $f\in N^{1,1}(X\setminus\partial\Omega)$, and
\begin{equation}\label{eq:Lebesgue point condition for f}
\lim_{r\to 0}\,\vint{B(x,r)}|f-f(x)|\,d\mu=0
\end{equation}
for $\hcal$-a.e.\@ $x\in\partial\Omega$. Then, $f\in N^{1,1}(X)$.
\end{proposition}
\begin{proof}
Fix $i\in\N$.
By the compactness of $\partial\Omega$, we find a covering $\{B_j=B(x_j,r_j)\}_{j=1}^M$
such that $r_j\le 1/i$ for all $j$, and
\begin{equation}\label{eq:covering of boundary of Omega}
\sum_{j=1}^M \frac{\mu(B(x_j,r_j))}{r_j}<\hcal(\partial\Omega)+1/i.
\end{equation}
Then, pick $2/r_j$-Lipschitz functions $\eta_j\in\Lip_c(B(x_j,2r_j))$ such that
$0\le \eta_j\le 1$ and $\eta_j=1$ on $B(x_j,r_j)$. 
Define
$v_i= \max_{j\in \{1,\ldots,M\}} \eta_j$.
Consider the function
\[
f_i\coloneq (1-v_i)f.
\]
Let $g\in L^1(X)$ be an upper gradient of $f$ in $X\setminus\partial\Omega$.
Clearly $2\ch_{2B_j}/r_j$ is an upper gradient of $\eta_j$, and then
$2\sum_{j=1}^M\frac{\ch_{2B_j}}{r_j}$ is an upper gradient of $v_i$.
We show that
\[
g_i\coloneq g+2\sum_{j=1}^M\frac{\ch_{2B_j}}{r_j}
\]
is a $1$-weak upper gradient of $f_i$ in $X$.
By the Leibniz rule, see e.g. \cite[Theorem 2.15]{BB}, $g_i$ is a $1$-weak upper gradient of $f_i$ in $X\setminus\partial\Omega$.
Take a curve $\gamma$ such that the upper gradient inequality is satisfied by $f_i$
and $g_i$ on all subcurves of $\gamma$ in $X\setminus\partial\Omega$;
this is true for $1$-a.e.\@ $\gamma$, by \cite[Lemma 1.34]{BB}.

Note that
\[
\dist\left(\partial\Omega,X\setminus\bigcup_{j=1}^M B_j\right)>0.
\]
Thus, $\gamma$ can be split into a finite number of subcurves each of which lies either entirely in $\bigcup_{j=1}^M B_j$, or entirely in $X\setminus \partial\Omega$.
If $\gamma_1$ is a subcurve lying entirely in $\bigcup_{j=1}^M B_j$,
\[
|f_i(\gamma_1(0))-f_i(\gamma_1(\ell_{\gamma_1}))|=|0-0|=0,
\]
so the upper gradient inequality is satisfied. If $\gamma_2$ is a subcurve lying entirely
in $X\setminus\partial\Omega$, then
\[
|f_i(\gamma_2(0))-f_i(\gamma_2(\ell_{\gamma_2}))|\le \int_{\gamma_2}g_i\,ds
\]
by our choice of $\gamma$. Summing over the subcurves, we obtain
\[
|f_i(\gamma(0))-f_i(\gamma(\ell_{\gamma}))|\le \int_{\gamma}g_i\,ds.
\]
Thus, $g_i$ is a $1$-weak upper gradient of $f_i$ in $X$.
By \eqref{eq:covering of boundary of Omega} we have
\[
\| g_i\|_{L^1(X)}\le \| g\|_{L^1(X)}
+2C_d(\hcal(\partial\Omega)+1/i),
\]
and thus $f_i\in N^{1,1}(X)$. Since Lipschitz functions are dense in $N^{1,1}(X)$, see e.g. \cite[Theorem 5.1]{BB}, we have also $f_i\in\BV(X)$, with
\[
\| Df_i\|(X)\le \| g_i\|_{L^1(X)}.
\]
Clearly $f_i\to f$ in $L^1(X)$ as $i\to\infty$.
By lower semicontinuity,
\[
\| Df\|(X)\le \liminf_{i\to\infty}\| Df_i\|(X)\le  \| g\|_{L^1(X)}
+2C_d\hcal(\partial\Omega).
\]
Thus, $f\in\BV(X)$. Recall the decomposition of the variation measure from \eqref{eq:decomposition}. Since $f\in N^{1,1}(X\setminus\partial\Omega)$, $\| Df\|^s(X\setminus\partial\Omega)=0$.
Since $\hcal(\partial\Omega)<\infty$, also $\| Df\|^c(\partial\Omega)=0$
by \cite[Theorem 5.3]{AMP}.
Finally, by \eqref{eq:decomposition},
\[
\| Df\|^j(\partial\Omega)\le C_d\int_{\partial\Omega\cap S_f}(f^{\vee}-f^{\wedge})\,d\hcal=0,
\]
since by the Lebesgue point condition \eqref{eq:Lebesgue point condition for f},
$f^{\wedge}(x)=f^{\vee}(x)\in\R$ for $\hcal$-a.e.\@ $x\in\partial\Omega$.
Thus, $\| Df\|^s(X)=0$, so that by \cite[Theorem 4.6, Remark 4.7]{HKLL}, there exists
a function $h\in N^{1,1}(X)$ with $\mu(\{h\neq f\})=0$.
By the Lebesgue point theorem \cite[Theorem 4.1, Remark 4.2]{KKST3} (note that this result assumes
that $\mu(X)=\infty$, but this assumption can be avoided by using~\cite[Lemma 3.1]{Mak}
instead of~\cite[Theorem 3.1]{KKST3} in the proof of Lebesgue point theorem found in~\cite{KKST3}),
\[
\lim_{r\to 0}\,\vint{B(x,r)}|f-h(x)|\,d\mu=\lim_{r\to 0}\,\vint{B(x,r)}|h-h(x)|\,d\mu=0
\]
for $1$-q.e. or equivalently $\hcal$-a.e. $x\in X$,
by \eqref{eq:null sets of Hausdorff measure and capacity}.
By the same Lebesgue point theorem, and \eqref{eq:Lebesgue point condition for f}, we have
\[
\lim_{r\to 0}\,\vint{B(x,r)}|f-f(x)|\,d\mu=0
\]
for $1$-q.e. $x\in X$. Thus, necessarily $h(x)=f(x)$ for $1$-q.e. $x\in X$.
Thus by \cite[Proposition 1.61]{BB}, $f\in N^{1,1}(X)$.
\end{proof}
\begin{corollary}\label{cor:extending from boundary to the whole space}
Let $\Omega\subset X$ be a bounded open set
whose boundary is codimension $1$ Ahlfors regular
as given in \eqref{eq:boundary codim Ahlfors regularity}.
Let $f\in L^1(\partial\Omega,\hcal)$ with $-1\le f\le 1$.
Then, there exists $Ef\in N^{1,1}(X)$  with $Ef(x)=f(x)$ for every $x\in\partial\Omega$. 
\end{corollary}
\begin{proof}
Combine Theorem \ref{thm:extension from boundary} and Proposition \ref{prop:extending to the boundary as Newtonian function}.
\end{proof}
In Example \ref{ex:nonexistence of solution}, it is crucial that $a>1$. If $f$ is
restricted in the same way as $u$,
solutions exist at least if $\Omega$ is sufficiently regular.
The proof relies on the following lower semicontinuity result,
which will also be used later in other contexts. Such a restriction is 
necessary even in Euclidean setting with smooth domains, see~\cite{MST}
and Example~\ref{ex:nonexistence of solution} (which, while is not a smooth domain, can be modified to be one).
\begin{lemma}\label{lem:lower semicontinuity of I}
Let $\Omega\subset X$ be a nonempty bounded open set satisfying the exterior measure density 
condition~\eqref{eq:exterior measure density condition},
and such that for any $u\in \BV(\Omega)$, the trace $Tu(x)$
exists for $\hcal$-a.e.\@ $x\in\partial^*\Omega$. Assume also that
$\partial\Om$ is codimension $1$ Ahlfors regular
as given in \eqref{eq:boundary codim Ahlfors regularity}.
Let $f\in L^1(\partial^*\Omega,P(\Om,\cdot))$ with $-1\le f\le 1$, satisfying
\eqref{eq:f integrates to zero}. Then if $E_i\subset \Omega$, $i\in\N$, are such that
$P(E_i,\Om)<\infty$ and $\ch_{E_i}\to \ch_E$ in $L^1(\Om)$,
it follows that
\[
I(\ch_E)\le \liminf_{i\to\infty} I(\ch_{E_i})\quad\textrm{and}\quad
I(-\ch_E)\le \liminf_{i\to\infty} I(-\ch_{E_i}).
\]
\end{lemma}
Note that since $\partial\Omega$ is compact, we have in particular $\hcal(\partial\Omega)<\infty$, and then $P(\Omega,X)<\infty$ (see e.g. \cite[Proposition 6.3]{KKST}).
\begin{proof}
By \eqref{eq:def of theta}, we have $f\in L^1(\partial^*\Omega,\mathcal H)$ and so we can extend
$f$ to a function $f\in L^1(\dOm,\mathcal H)$, e.g. by zero extension.
By Corollary \ref{cor:extending from boundary to the whole space}, there is an extension of $f$, still denoted simply by $f$, such that $f\in N^{1,1}(X)$.
We know that every function in the class $N^{1,1}(X)$ is $1$-quasicontinuous,
see~\cite[Theorem 1.1]{BBS} or \cite[Theorem 5.29]{BB}. Therefore by~\cite[Proposition 3.4]{BBM}
we know that for every $t\in\R$, $\{f>t\}$ and $\{f<t\}$ are $1$-quasiopen.
Then by~\cite[Lemma~2.3]{Fug71}, $\{t_1<f<t_2\}$ is also $1$-quasiopen
for any $t_1,t_2\in\R$.

Let $E\subset \Om$ such that $P(E,\Om)<\infty$. By Cavalieri's principle,
\begin{align}
\notag
&I(\ch_E)=P(E,\Omega)  + \int_0^1 \int_{\{f>t\}}T\ch_E\, dP(\Omega,\cdot)\,dt
- \int_0^1 \int_{\{f<-t\}}T\ch_E\, dP(\Omega,\cdot)\,dt\\
\notag
&\quad =\int_0^1 \Big[P(E,\Omega) \\
&\qquad\quad +  P(\Omega,\{T\ch_E=1\}\cap \{f>t\})
- P(\Omega,\{T\ch_E=1\}\cap \{f<-t\})\Big]\,dt.%
\label{eq:representation for functional of chi E}
\end{align}

Fix $t\in (0,1)$. Suppose $E_i\subset \Omega$, $i\in\N$, are such that $P(E_i,\Om)<\infty$
and $\ch_{E_i}\to \ch_E$ in $L^1(\Om)$ (and thus in fact in $L^1(X)$).
By Theorem \ref{thm:extendability of sets of finite perimeter}, also $P(E_i,X)<\infty$.
By lower semicontinuity and Lemma \ref{lem:estimate for perimeter in whole space}, we have
\[
P(E,X)\le\liminf_{i\to\infty}P(E_i,X)\le \liminf_{i\to\infty}I(\ch_{E_i})+2P(\Om,X),
\]
where we can assume the limit on the right-hand side to be finite. Thus $P(E,X)<\infty$.
By Proposition \ref{prop:lower semicontinuity in quasiopen sets}, we now have
\begin{align*}
 P(E&, \Omega\cap \{f>t\})+P(E,\partial\Omega\cap \{f>t\})\\
& = P(E,\{f>t\})\\
& \le \liminf_{i\to\infty} P(E_i,\{f>t\})\\
& =\liminf_{i\to\infty} \Big( P(E_i,\Omega\cap \{f>t\})+P( E_i,\partial\Omega\cap \{f>t\}) \Big).
\end{align*}
Thus, by Lemma \ref{lem:coincidence of perimeter of E and Omega},
\begin{align}
\notag &P(E,\Omega\cap \{f>t\})+P(\Omega,\{T\ch_E=1\} \cap\{f>t\})\\
&\quad\le \liminf_{i\to\infty} \Big( P(E_i,\Omega\cap \{f>t\})+P(\Omega,\{T\ch_{E_i}=1\} \cap\{f>t\})\Big).
\label{eq:first lower semicontinuity formula}
\end{align}
Since also $\ch_{\Omega\setminus E_i}\to \ch_{\Omega\setminus E}$ in $L^1(X)$, by the lower semicontinuity of perimeter we also get
\begin{align*}
P(\Omega& \setminus E,\Omega\cap \{f<-t\})+P(\Omega\setminus E,\partial\Omega\cap \{f<-t\})\\
& = P(\Omega\setminus E,\{f<-t\})\\
& \le \liminf_{i\to\infty} P(\Omega\setminus E_i,\{f<-t\})\\
& =\liminf_{i\to\infty} \Big( P(\Omega\setminus E_i,\Omega\cap \{f<-t\})+P(\Omega\setminus E_i,\partial\Omega\cap \{f<-t\}) \Big).
\end{align*}
Note that $T\ch_{\Omega\setminus E}(x)=1$ if and only if $T\ch_{E}(x)=0$. Thus
by Lemma \ref{lem:coincidence of perimeter of E and Omega},
\begin{align*}
&P(\Omega\setminus E,\Omega\cap \{f<-t\})+P(\Omega,\{T\ch_E=0\} \cap\{f<-t\})\\
&\qquad\le \liminf_{i\to\infty} \Big( P(\Omega\setminus E_i,\Omega\cap \{f<-t\})+P(\Omega,\{T\ch_{E_i}=0\} \cap\{f<-t\})\Big).
\end{align*}
By subtracting $P(\Omega,\{f<-t\})$ from both sides and noting that $P(F,A)=P(\Omega\setminus F,A)$ for any $\mu$-measurable $F\subset X$ and any set $A\subset \Omega$, we obtain
\begin{align}
\notag
P(E&,\Omega\cap \{f<-t\})-P(\Omega,\{T\ch_E=1\} \cap\{f<-t\})\\
&\le \liminf_{i\to\infty} \Big(P(E_i,\Omega\cap \{f<-t\})-P(\Omega,\{T\ch_{E_i}=1\} \cap\{f<-t\})\Big).
\label{eq:second lower semicontinuity formula}
\end{align}
By the fact that $P(E,\cdot)$ is a finite measure, for $\mathcal L^1$-a.e.\@ $t\in (0,1)$ we have
\[
P(E,\Omega\cap (\{f=t\}\cup\{f=-t\}))=0.
\]
For such $t$, by \eqref{eq:first lower semicontinuity formula}, and \eqref{eq:second lower semicontinuity formula} and using lower semicontinuity once more, in the $1$-quasiopen set $\{-t< f< t\}$,
\begin{align}
P&(E,\Omega)+ P(\Omega,\{T\ch_E=1\}\cap \{f>t\})- P(\Omega,\{T\ch_E=1\}\cap \{f<-t\}) \notag\\
&= P(E,\Omega\cap \{f>t\})+P(E,\Omega\cap \{f<-t\})+P(E,\Omega\cap \{-t<f<t\}) \notag\\
&\qquad +P(\Omega,\{T\ch_E=1\}\cap \{f>t\})- P(\Omega,\{T\ch_E=1\}\cap \{f<-t\}) \notag\\
&\le \liminf_{i\to\infty} \Big(P(E_i,\Omega\cap \{f>t\})+P(\Omega,\{T\ch_{E_i}=1\} \cap\{f>t\})\Big) \notag\\
&\qquad +\liminf_{i\to\infty} \Big(P(E_i,\Omega\cap \{f<-t\})-P(\Omega,\{T\ch_{E_i}=1\} \cap\{f<-t\})\Big) \notag\\
&\qquad +\liminf_{i\to\infty}P(E_i,\Omega\cap \{-t\le f\le t\}) \notag\\
&\le \liminf_{i\to\infty} \Big(P(E_i,\Omega) \notag\\
&\qquad  +P(\Omega,\{T\ch_{E_i}=1\} \cap\{f>t\})\Big)-P(\Omega,\{T\ch_{E_i}=1\} \cap\{f<-t\})\Big).
\label{eq:lower semicontinuity for level sets}
\end{align}

By combining \eqref{eq:representation for functional of chi E} and \eqref{eq:lower semicontinuity for level sets} and using Fatou's lemma, we obtain
\[
I(\ch_E)\le \liminf_{i\to\infty} I(\ch_{E_i}).
\]
Denoting $I(\cdot)=I_f(\cdot)$ to make the dependence on $f$ explicit, we have also
\[
I_f(-\ch_E)=I_{-f}(\ch_E)\le \liminf_{i\to\infty} I_{-f}(\ch_{E_i})=\liminf_{i\to\infty} I_f(-\ch_{E_i}),
\]
and thus the claim is proved.
\end{proof}
\begin{theorem}\label{thm:existence of solutions}
Let $\Om$ and $f$ be as in Lemma \ref{lem:lower semicontinuity of I}.
Then the restricted Neumann problem given on page~\pageref{restricted minimization problem}
has a solution.
\end{theorem}
\begin{proof}
Take a sequence $(u_i)\subset\BV(\Omega)$ with $-1\le u_i\le 1$ and
\[
I(u_i)<\inf_{v\in\BV(\Omega),\,\| v\|_{L^{\infty}(\Om)}
\le 1} I(v)+1/i\quad\textrm{for all }i\in\N.
\]
By Proposition \ref{prop:existence of set minimizers} we can assume that $u_i=\ch_{E_1^i}-\ch_{E_2^i}$ for 
disjoint sets $E_1^i,E_2^i\subset\Omega$, $i\in\N$.
By Lemma~\ref{lem:estimate for perimeter in whole space} and 
Lemma~\ref{lem:splitting into positive and negative parts},
\begin{align*}
P(E^i_1,X)\le  I(\ch_{E^i_1})+2P(\Omega,X)
&=  I(u_i)-I(-\ch_{E^i_2})+2P(\Omega,X)\\
&\le  I(u_i)+3P(\Omega,X),
\end{align*}
and similarly for the sets $E_2^i$.
We conclude that the sequences $P(E^i_1,X)$ and $P(E^i_2,X)$ are bounded, and so by
\cite[Theorem 3.7]{M} there are sets $E_1,E_2\subset\Om$ such that
$\ch_{E_1^i}\to \ch_{E_1}$ in $L^1(X)$ and $\ch_{E_2^i}\to \ch_{E_2}$ in $L^1(X)$,
passing to a subsequence if needed (without relabeling the sequences).
Then, clearly also $\mu(E_1\cap E_2)=0$.
By lower semicontinuity, $P(E_1,X)<\infty$ and $P(E_2,X)<\infty$.
Thus by Lemma \ref{lem:splitting into positive and negative parts}
and Lemma \ref{lem:lower semicontinuity of I},
\begin{align*}
I(\ch_{E_1}-\ch_{E_2})
&= I(\ch_{E_1})+I(-\ch_{E_2})\\
&\le \liminf_{i\to\infty} I(\ch_{E_1^i})+\liminf_{i\to\infty} I(-\ch_{E_2^i})\\
&\le \liminf_{i\to\infty} \left(I(\ch_{E_1^i})+ I(-\ch_{E_2^i})\right)\\
&=\liminf_{i\to\infty} \left(I(\ch_{E_1^i}-\ch_{E_2^i})\right)\ 
=\liminf_{i\to\infty} I(u_i).
\end{align*}
Thus, $\ch_{E_1}-\ch_{E_2}$ is a solution.
\end{proof}
\section{When $f\colon\partial\Om\to\{-1,0,1\}$}
\label{sec:when f is integer}
In this section, we always assume that $\Omega\subset X$ is a nonempty bounded domain
with $P(\Om,X)<\infty$, such that
for any $u\in \BV(\Omega)$, the trace $Tu(x)$
exists for $\hcal$-a.e.\@ $x\in\partial^*\Omega$.
We also assume that the boundary data $f\in L^1(\partial^*\Omega,P(\Om,\cdot))$
satisfies~\eqref{eq:f integrates to zero}, that is, $\int_{\partial^*\Om}f\, dP(\Om,\cdot)=0$.
To know that minimizers exist, we need $-1\le f\le 1$ as in the previous section, see also~\cite{MST}.

In light of the result from the previous sections that $\ch_{E_1}-\ch_{E_2}$ is 
a solution for some choice of $E_1, E_2\subset\Om$, we see that the
``relative outer normal derivative'' of the solution (in relation to the total variation 
of the function) is directed either entirely outward (i.e., $\partial_\eta u/\| Du\|=\pm1$ in the 
Euclidean setting) or has vanishing derivative. Thus, in the Euclidean setting,
if one is to make sense of $f$ as the relative outer normal derivative of the solution,
then the only permissible values
one has for $f$ are $0$, $1$, and $-1$. 
This section is dedicated to the
study of boundary behavior of solutions $\ch_{E_1}-\ch_{E_2}$ for such $f$.
\begin{definition}\label{defn:relative-bdry}
For $E\subset\Om$ of finite perimeter in $\Om$, let 
$\partial_E\Om$ denote the collection of all points $x\in\partial^*\Om$ for which
$T\ch_E(x)=1$.
\end{definition}
Suppose $E_1,E_2\subset\Om$ are disjoint sets such that $\ch_{E_1}-\ch_{E_2}$ solves the
restricted Neumann problem.
Note that
\[
 0\ge I(\ch_{E_1})=P(E_1,\Om)-P(\Om, \partial_{E_1}\Om\cap\{f=-1\})
 +P(\Om, \partial_{E_1}\Om\cap\{f=1\}).
\]
Therefore
\begin{equation}\label{eq:E1-vs-f}
P(E_1,\Om)+P(\Om, \partial_{E_1}\Om\cap\{f=1\})\le P(\Om, \partial_{E_1}\Om\cap\{f=-1\}).
\end{equation}
From Lemma~\ref{lem:I(E)=I(-E)}, we can conclude that
$I(\ch_{E_1}-\ch_{E_2})=2 I(\ch_{E_1})=2 I(-\ch_{E_2})$. If
$I(\ch_{E_1}-\ch_{E_2})\ne 0$, then $I(\ch_{E_1}-\ch_{E_2})<0$, and hence
$I(\ch_{E_1})<0$.
If $P(E_1,\Om)=0$, then by the facts that $X$ supports the relative isoperimetric inequality
\eqref{eq:relative isoperimetric inequality}
and $\Om$ is connected,
we must have either that $\mu(\Om\setminus E_1)=0$ or $\mu(E_1)=0$, from either of which
we would have that $I(\ch_{E_1})=0$. Thus we must have $P(E_1,\Om)>0$.
However, from this and \eqref{eq:E1-vs-f}  we can only infer that
\begin{equation}\label{E1-f=-1Weak}
P(\Om, \partial_{E_1}\Om\cap\{f=1\})<P(\Om, \partial_{E_1}\Om\cap\{f=-1\}).
\end{equation}
On the set $\partial_{E_1}\Om$ one should understand that the relative outer normal
derivative of $\ch_{E_1}-\ch_{E_2}$ must be $-1$; thus on the set
$\partial_{E_1}\Om\cap\{f=1\}$ the relative outer normal derivative of $\ch_{E_1}
-\ch_{E_2}$ does not agree with the boundary data $f=1$. The above inequality 
therefore implies that the relative outer normal derivative of $\ch_{E_1}-\ch_{E_2}$
agrees more often than not with the boundary data $f$ where $f\ne 0$. We would
prefer to obtain a better quantitative version of this statement.
\begin{proposition}
Suppose that $\Om$, as a metric measure space equipped with the measure $\mu\lfloor_\Om$, 
supports a $(1,1)$-Poincar\'e inequality and a measure density
condition: there is some $C\ge 1$ and $r_0>0$ such that 
\begin{equation}
\label{eq:int+ext_measure-density}
\mu(B(x,r)\cap\Om)\ge\frac{\mu(B(x,r))}{C}
\end{equation}
for every $x\in\partial\Om$ and $0<r<r_0$.
Suppose also that $\dOm$ is codimension $1$ Ahlfors regular
as defined in \eqref{eq:boundary codim Ahlfors regularity}.
Assume that $\emptyset \neq E_1\subsetneq \Om$ 
is such that $\ch_{E_1}-\ch_{\Om \setminus E_1}$ solves the restricted Neumann problem
with boundary data $f\colon \partial^*\Om \to \{-1, 0, 1\}$.
If $\mu(E_1)\le \mu(\Om \setminus E_1)$, then
\begin{equation}\label{eq:E_1-betterEst}
P(\Om,\partial_{E_1}\Om\cap\{f=1\})\le
\frac{C_\Om - 1}{C_\Om + 1}\,
  P(\Om,\partial_{E_1}\Om\cap\{f=-1\}),
\end{equation}
where the constant $C_\Om > 1$ is independent of $f$ and $E_1$. Otherwise,
\[
P(\Om,\partial_{\Om\setminus E_1}\Om\cap\{f=-1\})\le
\frac{C_\Om - 1}{C_\Om + 1}\,
  P(\Om,\partial_{\Om \setminus E_1}\Om\cap\{f=1\}).
\]
\end{proposition}
It is straightforward to check that \eqref{eq:int+ext_measure-density} can equivalently be
required for every $x\in\overline{\Om}$, possibly with different constants $C,r_0$.
Moreover, we will see that one can express $C_\Om = \|T\| \bigl(1+2C_{P_\Om}\diam(\Om)\bigr)$,
where $C_{P_\Om}>0$ is the constant associated with the Poincar\'e inequality on $\Om$ and
$\|T\|$ is the norm of the trace operator $T\colon \BV(\Om) \to L^1(\partial^*\Om,P(\Om,\cdot))$.
\begin{proof}
We will focus only on the situation when $\mu(E_1) \le \mu(\Om\setminus E_1)$. The other case can be proven analogously.
According to \cite[Theorem 5.5]{LaSh}, the trace operator
$T\colon \BV(\Om) \to L^1(\partial^*\Om,P(\Om,\cdot))$ is bounded, that is,
\begin{align*}
P(\Om,\partial_{E_1}\Om\cap\{f=1\})+P(\Om,\partial_{E_1}\Om\cap\{f=-1\})&\le 
\int_{\partial^*\Om}T\ch_{E_1}\, dP(\Om,\cdot)\\
  &\le \|T\| \bigl(\mu(E_1)+P(E_1,\Om)\bigr).
\end{align*}
The $(1,1)$-Poincar\'e inequality on $\Om$ yields that
\begin{equation}
 \label{eq:PI-for-Om}
 \frac{\mu(E_1)\mu(\Om\setminus E_1)}{\mu(\Om)}\le C_{P_\Om}  \diam(\Om) P(E_1,\Om),
\end{equation}
where $C_{P_\Om}>0$.
As $\mu(\Om \setminus E_1) \ge \mu(\Om)/2$ due to the assumed relation $\mu(E_1) \le \mu(\Om\setminus E_1)$, we obtain that $\mu(E_1)\le 2 C_{P_\Om} \diam(\Om) P(E_1,\Om)$. Therefore
\begin{align*}
P(\Om,\partial_{E_1}\Om\cap\{f=1\})+P(\Om,\partial_{E_1}\Om\cap\{f=-1\})& \\
\le \|T\| \bigl(1+2C_{P_\Om}\diam(\Om)\bigr) P(E_1,\Om) & \eqcolon C_\Om \, P(E_1,\Om).
\end{align*}
Consequently, we obtain from~\eqref{eq:E1-vs-f} that
\begin{align*}
P(\Om,\partial_{E_1}\Om\cap\{f&=1\})
  +P(\Om,\partial_{E_1}\Om\cap\{f=-1\})\\
&\le C_\Om\bigl(P(\Om,\partial_{E_1}\Om\cap\{f=-1\})-P(\Om,\partial_{E_1}\Om\cap\{f=1\})\bigr),
\end{align*}
which immediately implies the validity of~\eqref{eq:E_1-betterEst}.
\end{proof}
The inequality $\mu(E_1)\le \mu(\Om\setminus E_1)$ turns out to be crucial
when applying the estimate~\eqref{eq:PI-for-Om} to compare $\mu(E_1)$ with $P(E_1, \Om)$.
Otherwise, we cannot obtain \eqref{eq:E_1-betterEst} with a constant $C_\Om$ independent
of $E_1$, see Example \ref{exa:rectangle_bad-bdry-behavior}.

Nevertheless, we can define $C(E_1) = C_{P_\Om} \diam(\Om) \mu(\Om) / \mu(\Om \setminus E_1)$.
Then, \eqref{eq:PI-for-Om} leads to $\mu(E_1) \le C(E_1) P(E_1, \Om)$ and hence to the
quantitative estimate
\begin{equation}
  \label{eq:E_1-betterEst_largeE_1}
  P(\Om,\partial_{E_1}\Om\cap\{f=1\})
  <\frac{\|T\|(1+C(E_1)) - 1}{\|T\|(1+C(E_1)) + 1}\,P(\Om,\partial_{E_1}\Om\cap\{f=-1\}).
\end{equation}

If the $L^1$-boundedness of the trace operator is established by other means, we can
remove the assumptions of a $(1,1)$-Poincar\'e inequality for $\Om$ and
the measure density condition \eqref{eq:int+ext_measure-density}.
Then, we can bypass \eqref{eq:PI-for-Om} by
setting $C(E_1)=\mu(E_1)/P(E_1,\Om)$ to get \eqref{eq:E_1-betterEst_largeE_1}.

The following example shows that it is in general impossible to obtain an estimate better 
than~\eqref{E1-f=-1Weak} in case we wish the constants to be independent of $E_1$. On 
the other hand, the situation is different if $\partial \Om$ is of positive mean curvature in 
the sense of \cite{LMSS}, see Definition~\ref{defn:positive-curv} below. 
\begin{example}
\label{exa:rectangle_bad-bdry-behavior}
Fix $0<L<1/8$. Let $\Om=(0,1)^2$ be the
unit square in $\R^2$ (unweighted), and let $F_1\subset\partial\Om$ be given by the
union of the four line segments:
one connecting $(1-L,1)$ to $(1,1)$, one connecting $(1,1-L)$ to $(1,1)$, one 
connecting $(0,0)$ to $(0,\tfrac14)$, 
and one connecting $(0,\tfrac34)$ to $(0,1)$, the first two of which are each of 
length $L$ and the latter two of which are each of length $\tfrac14$. Let 
$F_2\subset\partial\Om$ be the union of three line segments, one connecting
$(0,\tfrac14)$ to $(0,\tfrac34)$ of length $\tfrac12$, and the other two,
each of length $L$, one connecting $(0,1)$ to $(L,1)$ and the other connecting
$(0,0)$ to $(L,0)$.
Let $f=\ch_{F_1}-\ch_{F_2}$. Now the restricted Neumann problem has exactly one solution, given by
$u=\ch_{E_1}-\ch_{E_2}$, where $E_1 = \Om \setminus E_2$ and $E_2$ is the triangular region
in $\Om$ with vertices $(1-L,1)$, $(1,1)$, and $(1,1-L)$. 
\end{example}
\begin{wrapfigure}{r}{71mm}
  \vspace{-15pt}
  \centering
  \includegraphics[width=65mm,height=46.5mm,page=1]{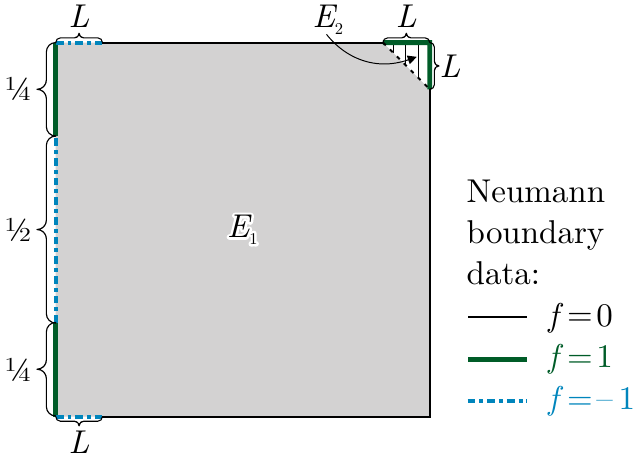}
\end{wrapfigure}

Using the above solution to the given Neumann problem, let us now show that it is in general impossible to obtain an estimate 
of the form~\eqref{eq:E_1-betterEst} in case $\mu(E_2)<\mu(E_1)$,
with a constant $C_\Om$ independent of $E_1$.
Indeed, we have
$P(\Om,\partial_{E_1}\Om\cap\{f=-1\})=2L + \tfrac12$
and
$P(\Om,\partial_{E_1}\Om\cap\{f=1\})=\tfrac12.$
Therefore, 
\[
\lim_{L\to 0}\frac{P(\Om,\partial_{E_1}\Om\cap\{f=-1\})}{P(\Om,\partial_{E_1}\Om\cap\{f=1\})}=1.
\]

\vspace{\topsep}

The example above heavily relies on the fact that the boundary data are non-zero on flat parts of 
$\dOm$. In the remaining part of this section, we will discuss the case when $\partial \Om$ is of 
positive mean curvature in the sense of \cite{LMSS}; see also \cite{SWZ}.
\begin{definition}\label{defn:positive-curv}
	Let
	$h\in \BV_{\loc}(X)$. We say that $u\in \BV_{\loc}(X)$ is a
	\emph{weak solution to the Dirichlet problem for least gradients in $\Om$ with
		boundary data $h$} if $u=h$ on $X\setminus\Om$ and
	\[
	\| Du\|(\overline{\Om})\le \| Dv\|(\overline{\Om})
	\]
	whenever $v\in \BV_{\loc}(X)$ with $v=h$ on $X\setminus\Om$.
\end{definition}
A weak solution exists whenever $h\in\BV_{\loc}(X)$ with $\| Dh\|(X)<\infty$,
see \cite[Lemma 3.1]{LMSS}.
\begin{definition}\label{def:positive-curve}
	We say that the boundary $\partial\Omega$
	has {\em positive mean curvature} if there exists a
	non-decreasing function $\varphi\colon (0,\infty)\to(0,\infty)$ and a constant $r_0>0$ such that for
	all $z\in\partial\Omega$ and all $0<r<r_0$ with $P(B(z,r),X)<\infty$,
	we have that
	$u^\vee \geq 1$ everywhere on $B(z,\varphi(r))$ for any weak solution $u$
	to the the Dirichlet problem for least gradients in $\Om$ with
	boundary data $\ch_{B(z,r)}$.
\end{definition}
Recall that the perimeter measure $P(E, \cdot)$  relates to $\hcal\lfloor_{\partial^*\!E}$
via the function $\theta_E\colon X \to [\alpha, C_d]$ as stated in \eqref{eq:def of theta}.
\begin{definition}[{\cite[Definition 6.1]{AMP}}]\label{def:local}
We say that $X$ is a \emph{local space} if, given any two sets of locally finite perimeter 
$E_1\subset E_2\subset X$, we have $\theta_{E_1}(x)=\theta_{E_2}(x)$ for $\hcal$-a.e.
$x\in \partial^*E_1\cap\partial^*E_2$.
\end{definition}
The assumption $E_1\subset E_2$ can in fact be dropped as shown in the discussion after \cite[Definition 5.9]{HKLL}.
See \cite{AMP} and \cite{L} for some examples of local spaces. See also
\cite[Example 5.2]{LaSh2} for an example of a space that fails to be local, despite
being equipped with a doubling measure that supports a Poincar\'e inequality.
\begin{theorem}
\label{thm:poscurv:bdry-data-agreement}
Suppose $X$ is a local space.
Assume that $\Om$ satisfies the exterior measure density 
condition~\eqref{eq:exterior measure density condition}, that $\mathcal H(\partial\Om)<\infty$,
and that $\partial\Om$ has 
positive mean curvature. Suppose that $\ch_{E_1}-\ch_{E_2}$ solves the restricted Neumann problem
with boundary data $f\colon \partial^*\Om\to \{-1,0,1\}$.
If $z\in\partial\Om$ such that $f=-1$ in a neighborhood of $z$, then $T\ch_{E_1}(z)=1$.

Moreover, if $u\in \BV(\Om)$ is any solution to the 
restricted Neumann problem with boundary data $f$
and $f=-1$ on $B(z,r)\cap\dOm$ for some $r>0$, then 
$u=1$ on $B(z,\varphi(r))\cap\Om$ and hence $Tu(z)=1$.
\end{theorem}
In the above, $r\mapsto\varphi(r)$ is the function associated with positive mean curvature
of $\partial\Om$ as in Definition~\ref{def:positive-curve}.
\begin{proof}
If $z\in\partial\Om$ such that $f=-1$ in a neighborhood of $z$,
we find $r>0$ such that $f=-1$ on $B(z,r)\cap \partial^*\Om$,
and $P(B(z,r),X)<\infty$ and $\mathcal H(\partial B(z,r)\cap \partial\Om)=0$;
the latter two facts hold for $\mathcal L^1$-a.e. $r>0$ by the $\BV$ coarea formula~\eqref{eq:coarea} 
and the fact that $\mathcal H(\partial\Om)<\infty$.
Take $K\subset X$ such that $\ch_K$ is a weak
solution to the Dirichlet problem for least gradients in $\Om$ with boundary data
$\ch_{B(z,r)}$; in particular, $\ch_K=\ch_{B(z,r)}$ on $X\setminus\Om$.
We let $E=E_1\cup (B(z,r)\setminus\Om)$
and claim that $K\cap E$ is another weak solution to the Dirichlet problem. Suppose it is not. Then
\[
P(K,\overline{\Om})< P(K\cap E,\overline{\Om}).
\]
By {\cite[Corollary~4.6]{LMSS}}, we have $T\ch_K(x)=\ch_{B(z,r)}(x)$ for $\mathcal H$-a.e.
$x\in \partial\Om$, and thus $\mathcal H(\partial^*K\cap \partial\Om)=0$, whence
$P(K,\partial\Om)=0$ by \eqref{eq:def of theta}.
Thus
\[
P(K,\overline{\Om})=P(K,\Om).
\]
Now we also have $T\ch_{K\cap E}\le \ch_{B(z,r)}$
$\mathcal H$-a.e.
on $\partial\Om$, and so $\mathcal H(\partial^*(K\cap E)\cap\partial\Om\setminus B(z,r))=0$.
Note that $P(E_1,X)<\infty$ by Theorem \ref{thm:extendability of sets of finite perimeter},
and then $P(K\cap E,X)<\infty$ by \cite[Proposition 4.7]{M}.
Thus by the fact that $P(K\cap E,\cdot)$ is a Borel outer measure and \eqref{eq:def of theta},
\begin{align*}
P(K\cap E&,\overline{\Om}) =P(K\cap E,\Om)+P(K\cap E,\partial\Om)\\
&=P(K\cap E,\Om)+P(K\cap E,\partial\Om\cap B(z,r))\\
&=P(K\cap E,\Om)+P(B(z,r)\setminus (K\cap E),\partial\Om\cap B(z,r))\\
&=P(K\cap E,\Om)+P(\Om,B(z,r)\cap \{T\ch_{K\cap E}=0\})\quad\text{by Lemma \ref{lem:coincidence of perimeter of E and Omega}}\\
&=P(K\cap E,\Om)+P(\Om,B(z,r)\setminus \partial_E \Om),
\end{align*}
since $T\ch_{K}=1$ $\mathcal H$-a.e. on $B(z,r)$. See Definition~\ref{defn:relative-bdry}
for the definition of $\partial_E\Om$.
Combining these,
\begin{equation}\label{eq:first comparison between K and E1}
 P(K,\Om)<P(K\cap E,\Om)+P(\Om, B(z,r)\setminus \partial_E \Om).
\end{equation}
It is straightforward to verify that
\[
\partial^*(K\cap E)\subset (\partial^*K\setminus O_{E})\cup (\partial^*E\cap I_{K}),
\]
where $I_K$ and $O_{E}$ stand for the measure theoretic interior and exterior, respectively, as defined by \eqref{eq:definition of measure theoretic interior} and
\eqref{eq:definition of measure theoretic exterior}.
By \eqref{eq:def of theta} and by $X$ being local, we obtain that
\begin{align*}
P(K\cap E,\Om)
& =\int_{\Om\cap\partial^*(K\cap E)}\theta_{K\cap E}\,d\hcal\\
& \le\int_{\Om\cap\partial^*K\setminus O_E}\theta_{K}\,d\hcal
+\int_{\Om\cap\partial^*E \cap I_K}\theta_{E}\,d\hcal\\
& = P(K,\Om\setminus O_E)+P(E,\Om\cap I_K).
\end{align*}
Combining this with \eqref{eq:first comparison between K and E1}, we get
\begin{equation}\label{eq:contra1}
P(K,\Om\cap O_E)<P(E,\Om\cap I_K)
   +P(\Om, B(z,r)\setminus \partial_E \Om).
\end{equation}
On the other hand, comparing $E$ against $E\cup K$ in the Neumann problem
(note that also $P(E\cup K,X)<\infty$ by \cite[Proposition 4.7]{M}),
by Lemma \ref{lem:both E1 and E2 are minimizers} we obtain
\begin{align}
\notag
P(E,\Om)&+\int_{\partial_{E}\Om} f\,dP(\Om,\cdot)
\le P(E\cup K,\Om)+\int_{\partial_{E\cup K}\Om} f\,dP(\Om,\cdot)\\
\label{eq:Neumann comparison of K and E1}
&= P(E\cup K,\Om)+\int_{\partial_{K}\Om\setminus \partial_{E}\Om} f\,dP(\Om,\cdot)
+\int_{\partial_{E}\Om} f\,dP(\Om,\cdot)\\
\notag
&= P(E\cup K,\Om)+\int_{B(z,r)\setminus \partial_{E}\Om} f\,dP(\Om,\cdot)
+\int_{\partial_{E}\Om} f\,dP(\Om,\cdot),
\end{align}
since we had $T\ch_K(x)=\ch_{B(z,r)}(x)$ for $\mathcal H$-a.e.
$x\in \partial\Om$.
Similarly as before,
it is straightforward to verify that
\[
\partial^*(E\cup K)
\subset (\partial^*E\setminus I_{K})\cup (\partial^*K\cap O_{E}).
\]
By \eqref{eq:def of theta} and the fact that $X$ is local, we now see that
\begin{align}
\notag P(E\cup K,\Om)
&=\int_{\Om\cap\partial^*(E\cup K)}\theta_{E\cup K}\,d\hcal\\
\notag
&\le\int_{\Om\cap \partial^*E \setminus I_K}\theta_{E}\,d\hcal
+\int_{\Om\cap\partial^*K\cap O_E}\theta_{K}\,d\hcal\\
&= P(E,\Om\setminus I_K)+P(K,\Om\cap O_E).
\label{eq:perimeter of K union E1}
\end{align}
Combining \eqref{eq:perimeter of K union E1} with \eqref{eq:Neumann comparison of K and E1} yields that
\begin{align*}
 P(E,\Om)-P(\Om&,\partial_{E}\Om\cap\{f=-1\})+P(\Om,\partial_{E}\Om\cap\{f=1\})\\
&\le P(E,\Om\setminus I_K)+P(K,\Om\cap O_E)
-P(\Om,B(z,r)\setminus \partial_{E}\Om)\\
&\quad+P(\Om,\partial_{E}\Om\cap \{f=1\})
-P(\Om,\partial_{E}\Om\cap \{f=-1\}).
\end{align*}
It follows that
\begin{equation}\label{eq:contra2}
P(E,\Om\cap I_K)
\le P(K,\Om\cap O_E)-P(\Om,B(z,r)\setminus \partial_{E}\Om).
\end{equation}
Since~\eqref{eq:contra1} is in contradiction with~\eqref{eq:contra2},
we have established the claim that $K\cap E$ is a weak solution
to the Dirichlet problem for least gradients in $\Om$ with boundary data $\ch_{B(z,r)}$.
Therefore, by the definition of positive mean curvature,
$B(z,\varphi(r))\subset K\cap E\subset E$
(up to a $\mu$-negligible set)
and in particular, $T\ch_{E_1}(z)=T\ch_{E}(z)=1$.

We complete the proof of this theorem by considering a solution $u$ for boundary data $f$
with $f=-1$ on $B(z,r)\cap\partial^*\Om$.
By the last part of
Proposition~\ref{prop:existence of set minimizers}, we can find two sequences
$t_{1,k},t_{2,k}\in (0,1)$
with $\lim_{k\to\infty}t_{1,k}=1$ and $\lim_{k\to\infty}t_{2,k}=1$ such that each
$\ch_{\{u>t_{1,k}\}}-\ch_{\{u<-t_{2,k}\}}$ is a solution to the same Neumann problem. Thus, by the above
argument, we have that $u\ge t_{1,k}$ in $B(z,\varphi(r))\cap\Om$ for each $k\in\N$, and thus the desired
conclusion follows by letting $k\to\infty$.
\end{proof}
In particular, it follows from the above result that \emph{every}
$z$ in the interior of the set $\{x\in\partial\Om :\, f(x)=-1\}$ satisfies $z\in\partial_{E_1}\Om$. Conversely,
$z\not\in\partial_{E_1}\Om$ whenever $z$ lies in the 
interior of the set $\{x\in\partial\Om :\, f(x)=1\}$.
\begin{remark}
Note that if $f=\ch_{F_1}-\ch_{F_2}$ with $F_1,F_2\subset\dOm$ disjoint, the above theorem gives
us good control over the solutions to the restricted Neumann problem with boundary data $f$ when
both $F_1$ and $F_2$ are relatively open subsets of $\partial\Om$. However, if $F_1$ and $F_2$ have
empty interior, the above theorem gives us no control over the solutions near the boundary.

Compare
this to the situation regarding the Dirichlet problem on domains whose boundary has positive mean curvature.
It is known that if the Dirichlet boundary data are continuous, then the solution to the
least gradient problem
on the domain has trace on the boundary that agrees with the boundary data,
see \cite{LMSS}. However, if the boundary data are not continuous, no such control over the trace of the 
solution is known except in special circumstances such as characteristic functions of relatively open
subsets $F\subset\partial\Om$ for which $\mathcal{H}(\partial\Om\cap\partial F)=0$. 
Indeed, in the Euclidean setting, with a Euclidean ball playing the role of the domain,
there are known to be boundary data, taken from the class $L^1$ of the boundary sphere, for which solutions
to the Dirichlet problem fail to have the correct trace, see~\cite{MRS}. 
\end{remark}
A natural question is whether we have any control over the solution near the part of the boundary where $f=0$. 
\begin{example}
Consider the simple example of $\Om=B(0,1) \subset \R^2$ (unweighted) with the boundary data
\[
  f(x,y) \coloneq \begin{cases}
     \sgn x & \text{for }(x,y)\in\dOm \text{ with } |x|\ge \frac12, \\
     0 & \text{otherwise.}
    \end{cases}
\]
We can easily see that it is impossible to determine what value
a solution $u$ will have near the boundary points 
where $f=0$. Indeed, the problem is solved by each of the following three functions:
\begin{align*}
  u_1(x,y) &= \ch_{(-1, 1/2)}(x) - \ch_{(1/2, 1)}(x) , \quad (x,y)\in\Om,  \\
  u_2(x,y) &= \ch_{(-1, -1/2)}(x) - \ch_{(-1/2, 1)}(x), \quad (x,y)\in\Om, \quad\text{and}\\
  u_3(x,y) &= \ch_{(-1, -1/2)}(x) - \ch_{(1/2, 1)}(x), \quad (x,y)\in\Om.
\end{align*}
Then, $Tu_1 \equiv 1$, $Tu_2 \equiv -1$, and $Tu_3 \equiv 0$ on the set $\{f=0\}$.
\end{example}
One might therefore wonder whether the zero Neumann data in a neighborhood of a boundary point 
guarantee that the solution is constant in a neighborhood of this point. In the following example, 
where a disk in the unweighted plane is discussed, we will see that such a conclusion indeed holds 
true. However, the subsequent two examples will prove the unweighted planar domain to be highly misleading.
\begin{example}\label{eq:disk-unweight}
Let $\Om = B(0,1)\subset \R^2$ (unweighted) 
and let $f\colon \dOm \to \{0, \pm1\}$. Let
$u = \ch_{E_1} - \ch_{\Om \setminus E_1} \in \BV(\Om)$ 
be a solution to the restricted Neumann problem with boundary data $f$.
We will now  show that if $z_0 \in \dOm$ lies in the interior of the set $\{f=0\}$,
then there is $r > 0$ such that $u$ is constant in $B(z_0, r)\cap\Om$.

Suppose for the sake of contradiction that $u$ is not constant on $B(z_0, r)$ for any $r>0$.
Fix $R>0$ such that $f(z)=0$ for all $z\in B(z_0, R)\cap\dOm$.
Since $\ch_{E_1}$ is a function of least gradient by Proposition \ref{prop:E1-leastGrad},
we can assume that $\partial E_1\cap\Om$ 
consists of straight line segments that connect points in $\partial\Om$
and do not intersect each other.
Consider the two closed half-disks whose union is $\overline\Om$ and whose intersection contains $z_0$. Take all the line 
segments of $\partial E_1$ that reach $B(z_0,R)\cap\partial\Om$
and lie within one of these half-disks. Then, move their 
end-points that lie within $B(z_0,R)\cap\dOm$ 
to $\partial B(z_0, R)\cap\dOm$ within the respective half-disk. Such a modification of $E_1$ will 
decrease the perimeter of $E_1$ inside $\Omega$ but the boundary integral will remain unchanged (since $f=0$ at all 
points where the trace of $\ch_{E_1}$ changed). In other words, such a modification will decrease the value of the 
functional $I(\cdot)$ and hence $u$ could not have been a solution.
\begin{center}
  \includegraphics[width=10cm,height=42mm,page=7]{NeumannP=1-Examples.pdf}

\vspace\topsep
\stepcounter{figure}
Figure \thefigure: The perimeter of $E_1$ inside $\Om$ is decreased by moving the endpoints of $\partial E_1$ from $\dOm \cap B(z_0, R)$ to $\dOm \cap \partial B(z_0, R)$.
\end{center}
\end{example}
Let us now consider a domain in 3-dimensional Euclidean space, where the situation turns out to 
be very different from the plane.
\begin{example}
Let $\Om = B(0,1) \subset \R^3$ (unweighted) and let
\[
  f(x,y,z) = \begin{cases}
  \sgn x & \text{when } |y|>\frac{1}{100};\\
  0 & \text{otherwise}.
  \end{cases}
  \hskip6cm  
\]
\end{example}
\begin{wrapfigure}{r}{60mm}
  \vspace{-64pt}
  \centering
  \includegraphics[width=57.4mm,height=54.1mm,page=8]{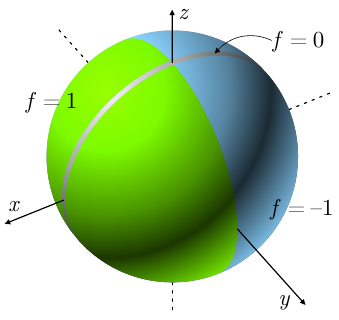}
  \vspace{-8pt}
\end{wrapfigure}
Based on Theorem~\ref{thm:poscurv:bdry-data-agreement}, the trace of a solution to the restricted minimization problem 
$u = \ch_{E_1} - \ch_{\Om \setminus E_1}$ necessarily attains the values of $-f$ in the region where $f\neq0$. Therefore, 
the set $E_1$ has to cover the surface of a unit half-ball with $x<0$, perhaps apart from the thin slit $|y|<\frac{1}{100}$. 
However, if $E_1$ consisted of at least 
two connected components, one for each component of the set $\{f=-1\}$, then the 
perimeter of $E_1$ inside $\Om$ would be greater than the perimeter of the half-ball $B(0,1) \cap \{x<0\}$, which equals 
the area of a unit disk $\{(0,y,z)\in\Om\}$. 
Hence, $E_1$ consists of a single connected component.

Then, $\partial E_1$ connects the two half-circles
on $\dOm$ with $x<0$ and $y = \frac{\pm 1}{100}$.
If the set $\widetilde{E}_1 \coloneq \{(x,y,z) \in \partial E_1:\, x<0, |y|<\frac{1}{100}\}$
lies entirely inside $\Om$, then the perimeter of this
portion of $\partial E_1$ can be bounded below by a 
half of the surface area of a cylinder of height
$\frac{2}{100}$ and radius $\bigl(1-(\frac{1}{100})^2\bigr){}^{1/2}$. Thus, the perimeter of $E_1$ inside $\Om$ will 
decrease if a sufficiently large part of $\widetilde{E}_1$ lies on~$\dOm$.
Therefore, the jump set of the trace of the solution $u$ has a nonempty intersection with the 
interior of the set $\{f=0\}$ and so the solution is nonconstant near the said intersection.

\vspace{\topsep} 
Next we show that the case of $\R^2$ equipped with an Ahlfors $2$-regular measure also
differs from the unweighted plane.
\begin{example}
\label{exa:weighted-disk}
Consider $X=\R^2$ endowed with the Euclidean distance and weighted Lebesgue measure $d\mu(z) \coloneq w(z)\,dz$, where
\[
  w(z) = \begin{cases}
    \frac12& \text{for } z \in [-\frac1{10}, \frac{1}{10}] \times [-\frac9{10}, \frac{9}{10}], \\
    1 & \text{otherwise.}
    \end{cases}
\]
Let $\Om = B(0,1)$ and define $f(x,y) = \sgn x$ for $(x,y)\in\partial\Om$ if $|x|> 1/\sqrt2$
and $f(x,y)=0$ otherwise. 
Considering $v(x,y) = -\sgn x$, $(x,y) \in \Om$, we obtain that
\[
  \inf_u I(u) \le I(v) = 2 \bigl(P(B_+(0,1), \Om) - P(\Om, \{f=1\})\bigr)
  = 2 \biggl(\frac{11}{10} - \frac{\pi}{2} \biggr) < 0,
\]
where $B_+(0,1)$ denotes the right half-disk $\{(x,y) \in B(0,1):\, x>0\}$. 
Observe also that the function $v$ is not actually a solution.

Let us now consider only candidates for solutions that are of least gradient in $\Om$ and
of the form $w = \ch_{E_1} - \ch_{E_2}$ 
such that the jump set of $w$ does not reach to the interior of the set $\{f=0\}$.
It is easy to verify for all $\alpha,\beta \in [-\frac{\pi}{4}, \frac{\pi}{4}]$ (and similarly for all 
$\alpha,\beta \in [\frac{3\pi}{4}, \frac{5\pi}{4}]$) that the
path of least weighted length that connects the boundary point 
$(\cos \alpha, \sin \alpha)$ with $(\cos \beta, \sin \beta)$ is a straight line segment. Thus, letting 
$w_0(x,y) = \ch_{(-1,-1/\sqrt2)}(x) - \ch_{(1/\sqrt2, 1)}(x)$ for $(x,y)\in\Om$, we have
$I(w_0) \le I(w)$, while
\[
  I(w_0) = 2 \biggl( P\Bigl(\Bigl\{(x,y)\in \Om:\, x>\frac{1}{\sqrt2}\Bigr\}, \Om\Bigr) - P\bigl(\Om,\{f=1\}\bigr) \biggr) 
  = 2 \biggl(\sqrt2 - \frac{\pi}{2} \biggr).
\]
In particular, $I(w) > I(v)$.

Thus, the jump set of a solution $u=\ch_{E_1} - \ch_{E_2}$
does reach to the interior of the set $\{f=0\}$, i.e., there is 
$z_0 \in \dOm$ and $r_0>0$ such that $f \equiv 0$ in $\dOm \cap B(z_0, r_0)$, but $u$ is not constant in $B(z_0, r)$ for any $r<r_0$.
It can be verified that 
\begin{align*}
E_1 & = \{(x,y)\in \Om:\, x < -\max\{0.1, |y|/9\}, \\
 E_2 & = \{(x,y)\in \Om:\, x > \max\{0.1, |y|/9\}.
\end{align*}
\end{example}
\section{Minimal solutions and their uniqueness}
\label{sec:minimal}
In this section, we assume that $\Omega\subset X$ is a nonempty bounded open set with
$P(\Om,X)<\infty$, such that for any $u\in \BV(\Omega)$, the trace $Tu(x)$
exists for $\hcal$-a.e.\@ $x\in\partial^*\Omega$.
We also assume that the boundary data $f\in L^1(\partial^*\Omega,P(\Om,\cdot))$
satisfies \eqref{eq:f integrates to zero}.

We saw in Example \ref{ex:nonuniqueness of solution} that solutions to the restricted Neumann
problem need not be unique. However, we will see in this section that \emph{minimal
solutions} exist and are unique.
\begin{lemma}\label{lem:strong subadditivity of I}
Let $E,K\subset \Om$ be of finite perimeter in $\Om$. Then
\[
I(\ch_{E\cap K})+I(\ch_{E\cup K})\le I(\ch_{E})+I(\ch_{K}).
\]
\end{lemma}
\begin{proof}
We have $P(E \cap K, \Om) + P(E \cup K, \Om) \le P(E, \Om) + P(K, \Om)$
by \cite[Proposition 4.7]{M}.
Then by linearity of traces, $\mathcal H$-a.e. on $\partial\Om$ we have
\[
  T\ch_{E \cap K} + T\ch_{E \cup K} = T(\ch_{E \cap K} + \ch_{E \cup K})
  = T(\ch_{E} + \ch_{K}) = T\ch_{E} + T\ch_{K}.
\]
The claim follows.
\end{proof}
\begin{definition}
A solution $u = \ch_{E_1} - \ch_{E_2}$ to the restricted Neumann problem
is said to be \emph{minimal} if whenever $\widetilde{E}_1,\widetilde{E}_2\subset \Om$
are disjoint sets such that $v = \ch_{\widetilde{E}_1} - \ch_{\widetilde{E}_2}$ is a solution, 
it follows that
$\mu(E_1\setminus \widetilde{E}_1)=0$
and $\mu(E_2\setminus \widetilde{E}_2)=0$.
\end{definition}
By Lemma \ref{lem:I(E)=I(-E)}, it is enough to compare with solutions of the form
$\ch_{\widetilde{E}} - \ch_{\Om\setminus \widetilde{E}}$.
\begin{lemma}
\label{lem:intersection-minimizes}
Suppose that $u_a = \ch_{E_a} - \ch_{\Om \setminus E_a}$ and $u_b= \ch_{E_b} - \ch_{\Om\setminus E_b}$ are both 
solutions to the restricted Neumann problem. Then, so are
\[
  u \coloneq \ch_{E_a \cap E_b} - \ch_{\Om \setminus (E_a \cap E_b)}\quad\text{and}\quad 
  v \coloneq \ch_{E_a \cup E_b} - \ch_{\Om \setminus (E_a \cup E_b)}.
\]
\end{lemma}
\begin{proof}
By Lemma \ref{lem:strong subadditivity of I} we know that
$I(\ch_{E_a \cap E_b}) + I(\ch_{E_a \cup E_b}) \le I(\ch_{E_a}) + I(\ch_{E_b})$.
By Lemma~\ref{lem:I(E)=I(-E)} we obtain that
$I(u) = 2 I(\ch_{E_a \cap E_b})$, $I(v) = 2 I(\ch_{E_a \cup E_b})$,
and analogously for $I(u_a)$ and $I(u_b)$ as well.
Then,
\begin{equation}
  \label{eq:Icapcup-est}
  \frac{I(u) + I(v)}{2} = I(\ch_{E_a \cap E_b}) + I(\ch_{E_a \cup E_b}) \le I(\ch_{E_a}) + I(\ch_{E_b}) = \frac{I(u_a) + I(u_b)}{2}\,.
\end{equation}
As $u_a$ and $u_b$ are solutions, we can estimate $I(u_a)=I(u_b) \le I(u)$ and
$I(u_a)\le I(v)$, which together with \eqref{eq:Icapcup-est} yields that
$I(u) = I(v) = I(u_a)$ and hence both $u$ and $v$ are solutions.
\end{proof}
\begin{theorem}
\label{thm:sol-with-minmeasure}
Assume that $\Om$ satisfies the exterior measure density condition
\eqref{eq:exterior measure density condition}, that $\partial\Om$ is codimension
$1$ Ahlfors regular as given in \eqref{eq:boundary codim Ahlfors regularity},
and that $-1\le f\le 1$.
Then there exists a unique (up to sets of $\mu$-measure zero)
minimal solution to the restricted Neumann problem.
\end{theorem}
\begin{proof}
	By Theorem \ref{thm:existence of solutions} we know that a solution exists.
	Let $\beta = \inf_E \mu(E)$, where the infimum is taken over all sets $E$ such that
	$u = \ch_E - \ch_{\Om\setminus E}$ is a solution.
	By Proposition \ref{prop:existence of set minimizers} and the fact that $\Om$ is bounded,
	$\beta<\infty$.
Let $\{E_k\}_{k=1}^\infty$ be a sequence of subsets of $\Om$ such that
$u_k = \ch_{E_k} - \ch_{\Om \setminus E_k}$ are solutions and
$\mu(E_k) \to \beta$. Let $\widetilde{E}_k = \bigcap_{j=1}^k E_j$. Then, all functions
\[
  v_k \coloneq \min_{1\le j \le k} u_j= \ch_{\widetilde{E}_k} - 
  \ch_{\Om \setminus \widetilde{E}_k}, \quad k=1,2,\ldots
\]
are also solutions by Lemma~\ref{lem:intersection-minimizes}.

Let $E_a = \bigcap_{k=1}^\infty E_k$.
Then $v_k \to \ch_{E_a} - \ch_{\Om \setminus E_a}$ in $L^1(\Om)$ and
$\mu(E_a) = \beta$.
By Lemma \ref{lem:lower semicontinuity of I} we obtain that
$I(\ch_{E_a} - \ch_{\Om \setminus E_a}) \le \liminf_{k\to\infty} I(v_k)$,
and so $\ch_{E_a} - \ch_{\Om \setminus E_a}$ is also a solution.

Now if $E\subset \Om$ is such that $\ch_E-\ch_{\Om\setminus E}$
is a solution, by Lemma \ref{lem:intersection-minimizes}
we know that $\ch_{E_a\cap E}-\ch_{\Om\setminus (E_a\cap E)}$ is also a solution,
and so $\mu(E_a\cap E)\ge \beta$.
Since $\mu(E_a)=\beta$, necessarily $\mu(E_a\setminus E)=0$.
By the same argument, we obtain that $E_a$ is the unique set with these properties,
up to sets of $\mu$-measure zero.

By an entirely analogous argument, we find a unique (up to sets of $\mu$-measure zero)
set $E_b\subset \Om$ such that
$\ch_{\Om \setminus E_b} - \ch_{E_b}$ is a solution, and whenever
$\ch_{\Om\setminus E}-\ch_E$ is another solution, then
$\mu(E_b\setminus E)=0$. By Lemma \ref{lem:both E1 and E2 are minimizers},
$\ch_{E_a}-\ch_{E_b}$ is the desired unique minimal solution.
\end{proof}
\section{Stability}
\label{sec:stability}
In this section, we always assume that $\Omega\subset X$ is a nonempty bounded open set with
$P(\Om,X)<\infty$, such that for any $u\in \BV(\Omega)$, the trace $Tu(x)$
exists for $\hcal$-a.e.\@ $x\in\partial^*\Omega$.

The goal here is to investigate stability of solutions to the restricted Neumann problem.
By stability we mean that
if a sequence of Neumann boundary data converges in $L^1(\partial^*\Om)$ to a function, then
the corresponding sequence of solutions converges
(perhaps up to a subsequence) to a solution to the Neumann problem with the 
limit boundary data. Stability properties give us a method by which we can, by hand, construct a solution to the Neumann
problem for complicated boundary data by using solutions to simpler boundary data.
\begin{example}
\label{exa:stability-issues}
Let $\Om = B(0,1) \subset \R^2\cong \C$ (unweighted). For each $k\in\N$, let
$\theta_k = \frac{\pi}{3} + \frac{(-1)^k}{k}$
and consider the sequence of boundary data functions
\[
  f_k(e^{i\theta}) \coloneq
  \begin{cases}
    1&\text{when } \theta \in (0, \theta_k] \cup [\pi-\theta_k, \pi), \\
    -1 & \text{when } \theta \in [-\theta_k, 0) \cup (\pi, \pi+\theta_k], \\
    0 & \text{otherwise}.
  \end{cases}
\]
It is easy to see that there are two types of minimal  solutions based on the 
value of $\theta_k$. If $\theta_k \in[\frac{\pi}{3}, \frac{\pi}{2}]$, then a solution can be expressed as $u(x,y) = -\sgn (y)$, 
which is also minimal in case $\theta_k > \frac{\pi}{3}$. However, if $\theta_k \in(0, \frac{\pi}{3}]$, then the minimal solution 
$u_k = \ch_{E_1^k} - \ch_{E_2^k}$ is determined by four disk segments whose arcs cover the connected components of 
$\{f_k \neq 0\}$, i.e.,
\begin{equation}
\label{eq:stability-ctrexample}
\begin{aligned}
  E_1^k & = \{(x,y)\in\Om:\, (1-\cos \theta_k) y \le  (|x|-1)\sin \theta_k\}, \\
  E_2^k & = \{(x,y)\in\Om:\, (1-\cos \theta_k) y \ge  (1-|x|)\sin \theta_k\}.
\end{aligned}
\end{equation}

\noindent
\begin{minipage}[t][][t]{0.495\linewidth}
\center
\includegraphics[width=50mm,height=51mm,page=3]{NeumannP=1-Examples.pdf}

The minimal solution $u_k$ if $\theta_k> \frac\pi3$

(also a solution if $\theta_k=\frac\pi3$)
\end{minipage}
\begin{minipage}[t][][t]{0.495\linewidth}
\center
\includegraphics[width=50mm,height=51mm,page=4]{NeumannP=1-Examples.pdf}

The minimal solution $u_k$ if $\theta_k\le \frac\pi3$
\end{minipage}

\vspace{1.5\topsep}
Thus, $u_{2k} = u$ for all $k=1,2,\ldots$, and trivially $u_{2k} \to u$ as $k\to \infty$. On the other hand 
$u_{2k+1} \to u_\infty = \ch_{E_1^\infty} - \ch_{E_2^\infty} \neq u$, where $E_1^\infty$ and $E_2^\infty$ are the 
sets as in~\eqref{eq:stability-ctrexample} for $\theta_\infty = \frac{\pi}{3}$.
Consequently, the sequence of solutions $\{u_k\}_{k=1}^\infty$ does not have any limit even though the sequence of 
boundary data functions converges in $L^1(\partial^*\Om,P(\Om,\cdot))$.
\end{example}
Note however that both functions $u$ and $u_{\infty}$ are solutions to the
restricted Neumann problem with boundary data given by 
$f=\lim_k f_k$. This observation suggests that a weaker notion of stability might apply here. Indeed, Theorem~\ref{thm:stable1} below will 
show that stability can be recovered if we allow for passing to a subsequence of the sequence of solutions.

In this section, we use the abbreviation $L^1(\partial^*\Om)\coloneq L^1(\partial^*\Om,P(\Om,\cdot))$.
\begin{lemma}\label{lem:compare-two-bdry-data}
If $u$ is a solution to the restricted Neumann problem with $L^1(\partial^*\Om)$-boundary data
$f$ and $v$ is a solution with $L^1(\partial^*\Om)$-boundary data $h$, then 
\[
|I_{f}(u)-I_{h}(v)|\le \| f-h\|_{L^1(\partial^*\Om)}.
\]
\end{lemma}
\begin{proof}
Note that $-1\le v\le 1$ and $-1\le u\le 1$. Therefore,
\[
I_{h}(u)-I_{f}(u)\le |I_{h}(u)-I_{f}(u)|=\biggl| \int_{\partial^*\Om}(f-h)\, Tu\, dP(\Om,\cdot)\biggr| \le \| f-h\|_{L^1(\partial^*\Om)}
\]
and
\[
I_{f}(v)-I_{h}(v)\le |I_{f}(v)-I_{h}(v)|=\biggl| \int_{\partial^*\Om}(f-h)\, Tv\, dP(\Om,\cdot)\biggr| \le \| f-h\|_{L^1(\partial^*\Om)}.
\]
It follows that
\[
I_f(u)\ge I_h(u)-\| f-h\|_{L^1(\partial^*\Om)}\ge I_h(v)-\| f-h\|_{L^1(\partial^*\Om)}
\]
and
\[
I_h(v)\ge I_f(v)-\| f-h\|_{L^1(\partial^*\Om)}\ge I_f(u)-\| f-h\|_{L^1(\partial^*\Om)}.
\]
In the above, we used the facts that $v$ is a solution for $I_h$ and that $u$ is a solution
for $I_f$. The desired conclusion now follows.
\end{proof}
\begin{theorem}\label{thm:stable1}
Assume that $\Om$ satisfies the exterior measure density condition
\eqref{eq:exterior measure density condition} and that $\partial\Om$ is codimension $1$
Ahlfors regular as given in \eqref{eq:boundary codim Ahlfors regularity}.
Assume that $f_k \colon \partial^*\Om \to [-1, 1]$ satisfy \eqref{eq:f integrates to zero},
that $f_k \to f$ in $L^1(\partial^*\Om)$ as $k\to\infty$, and that
$u_k = \ch_{E_1^k} - \ch_{E_2^k}$ are solutions to the
restricted Neumann problem with boundary data $f_k$, for disjoint sets $E_1^k,E_2^k\subset\Om$.
Then, there is a subsequence $\{u_{k_j}\}_{j=1}^\infty$ and a function
$u = \ch_{E_1} - \ch_{E_2}$ such that 
$u_{k_j} \to u$ in $L^1(\Om)$ and $u$ is a solution to the restricted Neumann problem
with boundary data $f$.
\end{theorem}
\begin{proof}
Clearly $f$ also satisfies \eqref{eq:f integrates to zero}.
By Theorem \ref{thm:existence of solutions}, we know that there exists a solution
$v\in\BV(\Om)$ for boundary data $f$. By Lemma~\ref{lem:compare-two-bdry-data},
$|I_f(v) - I_{f_k}(u_k)| \le \|f-f_k\|_{L^1(\partial^*\Om)} \to 0$ as $k\to \infty$.
By the fact that $u_k$ are solutions and Lemma \ref{lem:estimate for perimeter in whole space},
we get
\[
  \max\{ P(E_1^k, X), P(E_2^k, X)\} \le 2 P(\Om, X)\,.
\]
Thus by \cite[Theorem~3.7]{M}, there are sets $E_1,E_2\subset \Om$
such that $\ch_{E_1^k} \to \ch_{E_1}$ and 
$\ch_{E_2^k} \to \ch_{E_2}$ in $L^1(X)$, possibly having passed to a subsequence
(not relabeled). Define $u= \ch_{E_1} - \ch_{E_2}$. Then, by
the lower semicontinuity given in Lemma \ref{lem:lower semicontinuity of I},
\begin{align*}
  I_f(u)  & = I_f(\ch_{E_1}) + I_f(-\ch_{E_2})
  \le \liminf_{k\to\infty} I_f(\ch_{E^k_1}) + \liminf_{k\to\infty} I_f(-\ch_{E^k_2}) \\
  & \le \liminf_{k\to\infty} \bigl(I_{f_k}(\ch_{E^k_1}) + \|f-f_k\|_{L^1(\partial^*\Om)}\bigr) \\
  & \qquad  + \liminf_{k\to\infty} \bigl(I_{f_k}(-\ch_{E^k_2}) + \|f-f_k\|_{L^1(\partial^*\Om)}\bigr) \\
  & \le \liminf_{k\to\infty} \bigl(I_{f_k}(u_k) + 2 \|f-f_k\|_{L^1(\partial^*\Om)} \bigr) = I_f(v).
\end{align*}
Now, $v$ was a solution for $I_f$ and hence so is $u$.
\end{proof}
Observe that minimality of solutions need not be preserved when perturbing the boundary data. In 
Example~\ref{exa:stability-issues}, we saw that $u_{2k} \to u$ as $k\to\infty$, where $u$ was a solution for the limit 
boundary data. Nonetheless, while $u_{2k}$ were the minimal solutions for the respective boundary value problems, that 
was not the case for $u$, since the minimal solution for the limit boundary data was given by $u_\infty$.

In the example, the boundary data were given as $f_{2k} = \ch_{F_1^{2k}} - \ch_{F_2^{2k}}$ for decreasing sequences 
of sets $\{F_1^{2k}\}_{k=1}^\infty$ and $\{F_2^{2k}\}_{k=1}^\infty$.
One might also ask whether the minimality of a solution 
is preserved if the boundary data has the form $f_k = \ch_{F_1^k} - \ch_{F_2^k}$
for increasing sequences of sets 
$\{F_1^{k}\}_{k=1}^\infty$ and $\{F_2^{k}\}_{k=1}^\infty$. The next example shows that the minimality can be lost in this case as well.
\begin{example}
\label{exa:minimality-lost2}
Let $\Om = B(0,1) \subset \R^2 \cong \C$ (unweighted) and
\[
  f_k(e^{i\theta}) =
  \begin{cases}
    1&\text{when } \theta \in (\pi-\theta_k, \pi+\theta_k), \\
    -1 & \text{when } \theta \in (\frac{\pi}{3}-\theta_k,\frac{\pi}{3}) \cup (-\frac{\pi}{3},\theta_k - \frac{\pi}{3}), \\
    0 & \text{otherwise},
  \end{cases}
\]
where $\theta_k = \frac{\pi k}{3(k+1)}$.
Then, the minimal solutions are given by $u_k = \ch_{E_1^k} - \ch_{E_2^k}$, where 
$E_1^k = \{z \in \Om: \Re z > -\cos\theta_k\}$ and $E_2^k = \Om \setminus E_1^k$. The minimal solution for boundary data  given by the limit function $f_\infty$ is determined by the sets
$E_1 = \{z \in \Om: \Re z > \frac12 \}$ and 
$E_2 = \{z \in \Om: \Re z < \frac{-1}{2}\}$. In particular,
$E_1 \subsetneq \bigcap_k E_1^k = \Om \setminus E_2$.

\vspace{\topsep}
\noindent
\begin{minipage}[t][][t]{0.495\linewidth}
\center
\includegraphics[width=44mm,height=40mm,page=5]{NeumannP=1-Examples.pdf}

The minimal solution for $f_k$, $k\in\N$.
\end{minipage}
\begin{minipage}[t][][t]{0.495\linewidth}
\center
\includegraphics[width=44mm,height=40mm,page=6]{NeumannP=1-Examples.pdf}

The minimal solution for $f_\infty$.
\end{minipage}
\end{example}
In light of the above example, we give one explicit construction of a solution
(but not necessarily a minimal one) for limit boundary data. We first need the following
more general lemma.

In what follows, for $E\subset\Om$ of finite perimeter in $\Om$,
we denote
\[
I_f(E)\coloneq I_f(\ch_E)=\frac12I_f(\ch_{E}-\ch_{\Om\setminus E}).
\]
\begin{lemma}\label{lem:prelim-increase}
For each $k\in\N$, assume that
$f_k\in L^1(\partial^*\Om)$ satisfies \eqref{eq:f integrates to zero} and suppose that
$E_1^k, E_2^k\subset\Om$ are disjoint sets
such that $\ch_{E_2^k}-\ch_{E_1^k}$ is a solution to the restricted Neumann problem
with boundary data $f_k$.
Denote $E_k\coloneq E_1^k$.
Then for each $n\in\N$ and for each choice of
$k_1,\cdots,k_n\in\N$ with $k_1<\cdots<k_n$, we have 
\[
0\le I_{f_{k_n}}(E_{k_1}\cup\cdots\cup E_{k_n})-I_{f_{k_n}}(E_{k_n})
  \le 2\, \sum_{j=1}^{n-1}\| f_{k_j}-f_{k_{j+1}}\|_{L^1(\partial^*\Om)}.
\]
\end{lemma}
\begin{proof}
The first inequality follows from Lemma \ref{lem:both E1 and E2 are minimizers}. To prove the second, note first that
for $K\subset\Om$ of finite perimeter in $\Om$ and for $k\in\N$, we have that $I_{f_k}(E_k)\le I_{f_k}(E_k\cap K)$.
Moreover, by Lemma \ref{lem:strong subadditivity of I} we know that
\[
I_{f_k} (E_k\cup K)+I_{f_k}(E_k\cap K)\le
I_{f_k}(E_k)+I_{f_k}(K),
\]
and so
\begin{equation}\label{eq:1}
I_{f_k}(E_k\cup K)\le I_{f_k}(K).
\end{equation}
Furthermore, if $m\in\N$, then
\begin{equation}\label{eq:2}
I_{f_k}(K)\le I_{f_m}(K)+\| f_k-f_m\|_{L^1(\partial^*\Om)}.
\end{equation}
Now by an iterated ($(n-1)$-times) application of~\eqref{eq:1} followed by~\eqref{eq:2}, and
finally by
Lemma~\ref{lem:compare-two-bdry-data}, we obtain
\begin{align*}
I_{f_{k_n}}(E_{k_1}\cup\cdots\cup&\, E_{k_n})
\le I_{f_{k_n}}(E_{k_1}\cup\cdots \cup E_{k_{n-1}})\\
  &\le I_{f_{k_{n-1}}}(E_{k_1}\cup\cdots\cup E_{k_{n-1}})+\| f_{k_{n-1}}-f_{k_n}\|_{L^1(\partial^*\Om)}\\
  &\le \ldots\\
  &\le I_{f_{k_1}}(E_{k_1})+\sum_{j=1}^{n-1}\| f_{k_j}-f_{k_{j+1}}\|_{L^1(\partial^*\Om)}\\
  &\le I_{f_{k_n}}(E_{k_n})+\| f_{k_1}-f_{k_n}\|_{L^1(\partial^*\Om)}+\sum_{j=1}^{n-1}\| f_{k_j}-f_{k_{j+1}}\|_{L^1(\partial^*\Om)}\\
  &\le I_{f_{k_n}}(E_{k_n})+2\, \sum_{j=1}^{n-1}\| f_{k_j}-f_{k_{j+1}}\|_{L^1(\partial^*\Om)}.
\end{align*} 
Thus we obtain the desired inequality.
\end{proof}
\begin{theorem}\label{thm:monotone} 
Suppose that
$\Om$ satisfies the exterior measure density 
condition~\eqref{eq:exterior measure density condition}, and 
that $\partial\Om$ is codimension $1$ Ahlfors regular
as given in \eqref{eq:boundary codim Ahlfors regularity}.
For each $k\in\N$, suppose that $f_k\colon\partial^*\Om\to[-1,1]$ satisfies
\eqref{eq:f integrates to zero}, that $\| f_k-f_{k+1}\|_{L^1(\partial^*\Om)}\le 2^{-k}$,
and that $E_1^k, E_2^k\subset\Om$ are disjoint sets
such that $\ch_{E_1^k}-\ch_{E_2^k}$ is a solution for boundary data $f_k$. Set
$E_k=E_1^k$. Then the limit supremum
\[
E_+\coloneq \bigcap_{n\in\N}\bigcup_{k=n}^\infty E_k
\]
gives a solution $\ch_{E_+}-\ch_{\Om\setminus E_+}$ for the boundary data $f\coloneq \lim_k f_k$.
\end{theorem}
\begin{proof}
By Theorem \ref{thm:existence of solutions} we know that there exists a solution
$v\in\BV(\Om)$ for boundary data $f$.
From Lemma~\ref{lem:compare-two-bdry-data}
we see that
\begin{equation}\label{eq:v and Em}
| I_f(v)-2I_{f_m}(E_m) |\le \| f-f_m\|_{L^1(\partial^*\Om)}\to 0\quad\text{as }m\to\infty.
\end{equation}
For each $n\in\N$, set $K_n = \bigcup_{k=n}^\infty E_k$.
By Lebesgue's dominated convergence theorem,
$\ch_{E_n\cup\cdots\cup E_m}\to\ch_{K_n}$ in $L^1(\Om)$ as $m\to\infty$.
Fix $n\in\N$. By Lemma~\ref{lem:prelim-increase} and Lemma \ref{lem:lower semicontinuity of I},
\begin{align*}
\limsup_{m\to\infty} & \, I_{f_m}(E_m) \ge \limsup_{m\to\infty}I_{f_m}(E_n\cup\cdots\cup E_m)-2^{2-n}\\
  & \ge \liminf_{m\to\infty}I_{f}(E_n\cup\cdots\cup E_m)
  -\limsup_{m\to\infty}\| f-f_m\|_{L^1(\partial^*\Om)}-2^{2-n}\\
  &\ge I_f(K_n)-2^{2-n}.
\end{align*}
By letting $n\to\infty$ and recalling \eqref{eq:v and Em}, we obtain
\[
\frac{I_f(v)}{2}=\lim_{m\to\infty} I_{f_m}(E_m)
=\liminf_{n\to\infty} I_f(K_n)\ge I_f(E_+)
\]
by Lemma \ref{lem:lower semicontinuity of I}, since $\ch_{K_n}\to\ch_{E_+}$ in $L^1(\Om)$.
Thus $\ch_{E_+}$ is also a solution for boundary data $f$.
\end{proof}
It can be seen from Example~\ref{exa:minimality-lost2} that the set $E_+$ constructed in Theorem~\ref{thm:monotone} need 
not yield a \emph{minimal} solution to the limit Neumann problem.
\section{The unrestricted minimization problem}
In this section we always assume that $\Om$ is a nonempty bounded domain
with $P(\Om,X)<\infty$, such that the trace operator $T\colon\BV(\Om)\to L^1(\partial^*\Om, P(\Om,\cdot))$ is bounded.

So far we have looked at the most general situation where it is possible to have 
$I_f(u)=\| Du\|(\Om)+\int_{\partial^*\Om}Tu\,f\, dP(\Om,\cdot)<0$ for some $u\in \BV(\Om)$. 
To overcome the fact that should $I_f(u)<0$ for some $u$ then the minimal value of $I_f$ is $-\infty$, we 
considered minimization only over $u\in \BV(\Om)$ for which $-1\le u\le 1$.
In the special case where
\[
\inf_{u\in \BV(\Om)}I_f(u)\ge 0,
\]
the minimal energy must necessarily be $0$; hence constant functions (and in particular, the zero function) would be a
solution to the given Neumann boundary value problem with boundary data $f$. In this case we do not here need to
restrict our attention to $-1\le u\le 1$ alone, but to \emph{all} functions in the class $\BV(\Om)$. In this case it would be interesting
to see under what conditions we would have \emph{nonconstant} minimizers of $I_f$ exist. If there is one, then there are
infinitely many distinct (in the sense that they do not differ only by a constant) minimizers,
as seen by multiplying by a scalar. In this study we take inspiration 
from~\cite{Mo17}. We do not have a criterion that guarantees existence of a nonconstant minimizer. In the Euclidean setting,
the PDE approach helps in forming such a guarantee, but we do not have such an approach in the metric setting. However,
we do obtain a criterion under which there is no nonconstant minimizer, see Proposition~\ref{prop:existence} below.
As a consequence of Proposition~\ref{prop:scaling}
we also obtain that if there are no minimizers for the unrestricted problem for the
boundary data $f$,
then there is a positive number $\lambda(-f)$ such that the boundary data $\lambda(-f) f$ does have a minimizer.

From now on, let $g\in L^\infty(\partial^*\Om, P(\Om,\cdot))$ with
$\int_{\partial^*\Om}g\,dP(\Om,\cdot)=0$. We set 
$\mathcal{M}_g$ to be the collection of all functions $u\in \BV(\Om)$ such that
$\int_\Om u\, d\mu=0$ and
$\int_{\partial^*\Om}Tu\,g\, dP(\Om, \cdot)=1$, and 
\[
\lambda(g)\coloneq \inf_{u\in\mathcal{M}_g}\| Du\|(\Om).
\]
Note that if $\lambda(g)<1$, then there is some $u\in \BV(\Om)$ such that $I_{-g}(u)<0$, and hence
the unrestricted minimization problem for $f=-g$ has no solution.

\begin{proposition} \label{prop:existence}
If $\lambda(g)\ge 1$, then there is a solution to the unrestricted minimization problem for the energy $I_{-g}$ on $\Om$. Furthermore,
if $\lambda(g)>1$ then the only minimizers are constant functions.
\end{proposition}

\begin{proof}
We will prove the claim of the proposition by showing that for each $w\in\BV(\Om)$, we have $I_{-g}(w)\ge 0$.

For $w\in\BV(\Om)$, we have two possibilities. 
The first possibility is that $-\int_{\partial^*\Om}Tw\,g\, dP(\Om,\cdot)\ge 0$; in this case we have that $I_{-g}(w)\ge 0$.
Thus it suffices to consider only the case that 
$-\int_{\partial^*\Om}Tw\,g\, dP(\Om,\cdot)<0$.
In this case we set
\[
\alpha(w) = \int_{\partial^*\Om}Tw\,g\, dP(\Om,\cdot),
\]
and note that $\alpha(w)>0$. With $c:=\int_\Om w\, d\mu$, we have
$\alpha(w)^{-1}(w-c)\in \mathcal{M}_g$, and so by the hypothesis of the proposition,
\begin{equation}\label{eq:lambda}
\alpha(w)^{-1}\| Dw\|(\Om)\ge \lambda(g)\ge 1, 
\end{equation}
that is,
\[
\| Dw\|(\Om)\ge \alpha(w)=\int_{\partial^*\Om}Tw\,g\, dP(\Om,\cdot).
\] 
It then follows that 
\[
I_{-g}(w)=\| Dw\|(\Om)+\int_{\partial^*\Om}Tw\, [-g]\, dP(\Om,\cdot)\ge 0.
\]
Finally, suppose that $\lambda(g)>1$. If $-\int_{\partial^*\Om}Tw\,g\, dP(\Om,\cdot)<0$, then~\eqref{eq:lambda} implies that
\[
\| Dw\|(\Om)\ge \lambda(g)\, \int_{\partial^*\Om}Tw\,g\, dP(\Om,\cdot)>\int_{\partial^*\Om}Tw\,g\, dP(\Om,\cdot),
\]
and hence $I_{-g}(w)>0$. On the other hand, if $-\int_{\partial^*\Om}Tw\,g\, dP(\Om,\cdot)\ge 0$, then 
$I_{-g}(w)\ge \| Dw \|(\Om)$. Since the least value of the functional $I_{-g}$ is zero, its minimizer $w\in\BV(\Om)$ satisfies $\| Dw\|(\Om)=0$, that is,
$w$ is constant.
\end{proof}

From the above proposition, it follows that if there is a nonconstant minimizer for $I_{-g}$, then necessarily
$\lambda(g)=1$. Observe that if $\tau$ is a positive real number, then $\lambda(\tau g)=\lambda(g)/\tau$.
Thus if $\lambda(g)>0$, then $I_{-\lambda(g) g}$ does have a minimizer from $\BV(\Om)$.

\begin{proposition} \label{prop:scaling}
Suppose the trace operator $T\colon\BV(\Om)\to L^1(\partial^*\Om, P(\Om,\cdot))$ is surjective
and that there is
a constant $C>0$ such that whenever
$u\in\BV(\Om)$ with $\int_\Om u\, d\mu=0$, we have
\[
\int_{\partial^*\Om}|Tu|\, dP(\Om,\cdot)\le C \| Du\|(\Om).
\]
If $\| g\|_{L^\infty(\partial^*\Om, P(\Om,\cdot))}>0$, then $\lambda(g)>0$.
\end{proposition}

We refer the interested reader to~\cite[Theorem~5.5]{LaSh} together with~\cite[Theorem~1.2]{MSS}
for geometric conditions on $\Om$ that 
guarantee that the hypotheses of the above proposition hold. Note that if $T$ is a bounded operator in
the sense of~\cite[Theorem~5.5]{LaSh}, then by the Poincar\'e inequality on $\Om$ we obtain the control of
$\int_{\partial^*\Om}|Tu|\, dP(\Om,\cdot)$ solely in terms of $\| Du\|(\Om)$ for $u\in\BV(\Om)$ with $\int_\Om u\, d\mu=0$.

\begin{proof}
Since $\| g\|_{L^\infty(\partial^*\Om, P(\Om,\cdot))}>0$, the class $\mathcal{M}_g$ is non-empty. Indeed,
we can choose a function $w\in\BV(\Om)$ such that $T w=g$, and then 
$0<\alpha(w)= \int_{\partial^*\Om} Tw\,g\, dP(\Om,\cdot)<\infty$. With $c:=\int_\Om w\, d\mu$, 
 we have $\alpha(w)^{-1}(w-c)\in\mathcal{M}_g$.

Now suppose that $w\in \mathcal{M}_g$. As $\int_\Om w\, d\mu=0$, 
we have by assumption
\[
\int_{\partial^*\Om}|Tw|\, dP(\Om,\cdot)\le C \| Dw\|(\Om).
\]
Hence,
\begin{align*}
1= \int_{\partial^*\Om}Tw\,g\, dP(\Om,\cdot) 
  &\le \| g\|_{L^\infty(\partial^*\Om,P(\Om,\cdot))}\int_{\partial^*\Om}|Tw|\, dP(\Om,\cdot)\\
  &\le C \| g\|_{L^\infty(\partial^*\Om,P(\Om,\cdot))}\| Dw\|(\Om).
\end{align*}
Thus we must have
\[
\lambda(g)\ge \frac{1}{C \| g\|_{L^\infty(\partial\Om,P(\Om,\cdot))}}>0.
\qedhere
\]
\end{proof}
\noindent Address:

\vskip1ex plus 1ex minus 0.5ex
\noindent Department of Mathematical Sciences\\
P. O. Box 210025\\
University of Cincinnati\\
Cincinnati, OH 45221-0025\\
U.S.A.

\vskip1ex plus 1ex minus 0.5ex
\noindent E-mail:\\
{\tt panu.lahti@aalto.fi}, {\tt lukas.maly@matfyz.cz}, {\tt shanmun@uc.edu}
\end{document}